\documentclass[a4paper,11pt,reqno]{amsart}
\usepackage{amssymb,mathrsfs}

\newtheorem{theorem}{Theorem}[section]
\newtheorem{proposition}[theorem]{Proposition}
\newtheorem{lemma}[theorem]{Lemma}
\newtheorem{corollary}[theorem]{Corollary}

\theoremstyle{definition}
\newtheorem{definition}[theorem]{Definition}
\newtheorem{notation}[theorem]{Notation}

\theoremstyle{remark}
\newtheorem{remark}[theorem]{Remark}
\newtheorem{example}[theorem]{Example}

\newcommand{\bA}{\mathbf{A}}
\newcommand{\bN}{\mathbf{N}}
\newcommand{\bP}{\mathbf{P}}
\newcommand{\bQ}{\mathbf{Q}}
\newcommand{\bR}{\mathbf{R}}
\newcommand{\bZ}{\mathbf{Z}}
\newcommand{\cF}{\mathcal{F}}
\newcommand{\cS}{\mathcal{S}}
\newcommand{\fa}{\mathfrak{a}}
\newcommand{\fb}{\mathfrak{b}}
\newcommand{\fc}{\mathfrak{c}}
\newcommand{\fd}{\mathfrak{d}}
\newcommand{\fe}{\mathfrak{e}}
\newcommand{\fm}{\mathfrak{m}}
\newcommand{\fn}{\mathfrak{n}}
\newcommand{\fp}{\mathfrak{p}}
\newcommand{\fA}{\mathfrak{A}}
\newcommand{\sI}{\mathscr{I}}
\newcommand{\sJ}{\mathscr{J}}
\newcommand{\sO}{\mathscr{O}}
\newcommand{\sfa}{\mathsf{a}}
\newcommand{\sfb}{\mathsf{b}}
\newcommand{\sfc}{\mathsf{c}}
\newcommand{\sfd}{\mathsf{d}}
\newcommand{\sfe}{\mathsf{e}}
\newcommand{\sfA}{\mathsf{A}}
\newcommand{\sfI}{\mathsf{I}}
\newcommand{\sfJ}{\mathsf{J}}
\newcommand{\sfP}{\mathsf{P}}
\newcommand{\sfQ}{\mathsf{Q}}
\newcommand{\abs}[1]{{|{#1}|}}
\newcommand{\rd}[1]{{\lfloor{#1}\rfloor}}
\newcommand{\bracketsup}[1]{\textup{(}#1\textup{)}}

\DeclareMathOperator{\mld}{mld}
\DeclareMathOperator{\ord}{ord}
\DeclareMathOperator{\Proj}{Proj}
\DeclareMathOperator{\Spec}{Spec}
\DeclareMathOperator{\wt}{wt}

\begin{document}
\title{Minimal log discrepancies on smooth threefolds}

\author{Masayuki Kawakita}
\address{Research Institute for Mathematical Sciences, Kyoto University, Kyoto 606-8502, Japan}
\email{masayuki@kurims.kyoto-u.ac.jp}
\thanks{Partially supported by JSPS Grant-in-Aid for Scientific Research (C) 19K03423.}

\begin{abstract}
We completely prove the ACC for minimal log discrepancies on smooth threefolds. It implies on smooth threefolds the ACC for $a$-lc thresholds, the uniform $\fm$-adic semi-continuity of minimal log discrepancies and the boundedness of the log discrepancy of some divisor that computes the minimal log discrepancy.
\end{abstract}

\maketitle

\section{Introduction}
The minimal log discrepancy, introduced by Shokurov, is a fundamental invariant of singularity in birational geometry. It defines the class of singularities in the minimal model program. In addition, the termination of flips is reduced to two conjectural properties, the ACC and the lower semi-continuity, of minimal log discrepancies \cite{Sh04}. In contrast to the importance, the minimal log discrepancy is considered to be very hard to deal with. In fact, the ACC is still unknown in dimension three and it is one of the most important remaining problems in the minimal model theory of threefolds.

In our preceding paper \cite{K21}, we confirmed on a fixed germ $x\in X$ of a klt variety the equivalence of Shokurov's ACC conjectures for minimal log discrepancies and for $a$-lc thresholds, Musta\c{t}\u{a}'s uniform $\fm$-adic semi-continuity conjecture for minimal log discrepancies and Nakamura's boundedness conjecture for the log discrepancy of some divisor that computes the minimal log discrepancy. The purpose of the present paper is to establish all of them on smooth threefolds.

For an $\bR$-ideal $\fa=\prod_{j=1}^e\fa_j^{r_j}$ on $X$, we write $a_E(X,\fa)$ for the log discrepancy of a divisor $E$ over $X$ with respect to the pair $(X,\fa)$ and write $\mld_x(X,\fa)$ for the minimal log discrepancy of $(X,\fa)$ at $x$. For a subset $I$ of the positive real numbers, we write $\fa\in I$ if the exponents $r_j$ in $\fa$ belong to $I$. ACC stands for the ascending chain condition whilst DCC stands for the descending chain condition.

\begin{theorem}[ACC for minimal log discrepancies]\label{thm:main}
Fix a subset $I$ of the positive real numbers which satisfies the DCC\@. Then the set
\[
\{\mld_x(X,\fa)\mid\textup{$x\in X$ a smooth threefold},\ \textup{$\fa$ an $\bR$-ideal},\ \fa\in I\}
\]
satisfies the ACC\@.
\end{theorem}

\begin{theorem}[ACC for $a$-lc thresholds]\label{thm:alc}
Fix a non-negative real number $a$ and a subset $I$ of the positive real numbers which satisfies the DCC\@. Then the set
\[
\biggl\{t\in\bR\;\biggm|
\begin{array}{l}
\textup{$x\in X$ a smooth threefold},\ \textup{$\fa$, $\fb$ $\bR$-ideals},\\
t\ge0,\ \mld_x(X,\fa\fb^t)=a,\ \fa\in I,\ \fb\in I
\end{array}
\biggr\}
\]
satisfies the ACC\@.
\end{theorem}

\begin{theorem}[Uniform $\fm$-adic semi-continuity]\label{thm:adic}
Fix a subset $I$ of the positive real numbers which satisfies the DCC\@. Then there exists a positive integer $l$ depending only on $I$ such that for $\bR$-ideals $\fa=\prod_{j=1}^e\fa_j^{r_j}$ and $\fb=\prod_{j=1}^e\fb_j^{r_j}$ on the germ $x\in X$ of a smooth threefold, if $r_j\in I$ and $\fa_j+\fm^l=\fb_j+\fm^l$ for all $j$, where $\fm$ is the maximal ideal in $\sO_X$ defining $x$, then $\mld_x(X,\fa)=\mld_x(X,\fb)$.
\end{theorem}

\begin{theorem}[Boundedness]\label{thm:nakamura}
Fix a subset $I$ of the positive real numbers which satisfies the DCC\@. Then there exists a positive integer $l$ depending only on $I$ such that for an $\bR$-ideal $\fa$ on the germ $x\in X$ of a smooth threefold, if $\fa\in I$, then there exists a divisor $E$ over $X$ which computes $\mld_x(X,\fa)$ and has $a_E(X)\le l$.
\end{theorem}

The $\fm$-adic semi-continuity conjecture and the boundedness conjecture originally assumed that $I$ is finite. The extension to the case when $I$ satisfies the DCC is due to Han, Liu and Luo \cite{HLLaxv} and it will be explained as Proposition~\ref{prp:HLL}.

Once the theorems are established, it is relatively simple to extend them to the statements on a fixed terminal quotient threefold singularity as in Corollary~\ref{crl:quot}. Here we state a generalisation of Theorem~\ref{thm:main}.

\begin{corollary}\label{crl:main}
Fix a positive integer $r$ and a subset $I$ of the positive real numbers which satisfies the DCC\@. Then the set
\[
\biggl\{\mld_x(X,\fa)\;\biggm|
\begin{array}{l}
\textup{$x\in X$ a terminal quotient threefold singularity}\\
\textup{of index at most $r$},\ \textup{$\fa$ an $\bR$-ideal},\ \fa\in I
\end{array}
\biggr\}
\]
satisfies the ACC\@.
\end{corollary}

The ACC for minimal log discrepancies is predicted without restriction to smooth varieties. In dimension three, it has been proved for sequences of minimal log discrepancies with limit at least one \cite{HLLaxv}. In the case when the threefolds have no boundary, it is known for sequences with limit greater than $5/6$ \cite{LLaxv} after the works \cite{Ji21}, \cite{LX21}.

We shall outline the proof. We seek Theorem~\ref{thm:nakamura} amongst the four equivalent statements. Our prior work \cite{K21} reduced it to the case when the boundary splits into a canonical part and the maximal ideal $\fm$ to some power. To be precise, for fixed rational numbers $q$ and $s$, it suffices to bound the log discrepancy of some divisor that computes $\mld_x(X,\fa^q\fm^s)$ for an ideal $\fa$ such that $\mld_x(X,\fa^q)$ equals one. The boundedness has been settled when the log canonical threshold of $\fm$ on $(X,\fa^q)$ is at most one-half. In the remaining case, the pair $(X,\fa^q)$ is canonical of semistable type (Definition~\ref{dfn:semistable}), which means that the maximal ideal has order one along every divisor that computes $\mld_x(X,\fa^q)=1$.

We begin with a study of composition of weighted blow-ups. Let $E\subset Y\to x\in X$ be the weighted blow-up with $\wt(x_1,x_2,x_3)=(w_1,w_2,1)$ for $w_2\le w_1$ which is crepant in the sense that $a_E(X,\fa^q)$ is one. Consider the composite with a weighted blow-up $F\subset Z\to y\in Y$ of similar type with $\wt(y_1,y_2,x_3)=(v_1,v_2,1)$ for $v_2\le v_1$ at a point $y$ outside the strict transform of the divisor defined by $x_3$. In general, the divisor $F$ over $X$ is not obtained by any weighted blow-up of $X$ (Example~\ref{exl:composite}). However, provided that the slope $v_1/v_2$ is almost less than $w_1/w_2$, $F$ is realised by some weighted blow-up (Theorem~\ref{thm:composite}).

The problem is formulated in terms of the generic limit of ideals. The generic limit $\sfa$ of a sequence $\{\fa_i\}_i$ of ideals on $x\in X$ is defined on the spectrum $\hat x\in \hat X$ of the completion of an extension of the local ring $\sO_{X,x}$. After our former work \cite{K15}, it remains to deal with the case when $(\hat X,\sfa^q)$ has the smallest lc centre of dimension one (Theorem~\ref{thm:special}). In this case, a certain weighted blow-up of $\hat X$ with $\wt(x_1,x_2)=(w_1,w_2)$ produces a divisor $\hat F$ such that $a_{\hat F}(\hat X,\sfa^q)$ is zero. It comes from a sequence of crepant weighted blow-ups of $X$ with $\wt(x_{1i},x_{2i},x_{3i})=(iw_1,iw_2,1)$. Conversely if there exists a sequence of crepant weighted blow-ups of $X$ with $\wt(x_1,x_2,x_3)=(w_{1i},w_{2i},1)$ such that $w_{2i}$ diverges to infinity and such that the slope $w_{1i}/w_{2i}$ converges close to $w_1/w_2$, then $a_{\hat F}(\hat X,\sfa^q)$ is zero (Propositions~\ref{prp:down},~\ref{prp:up}).

Given a sequence $\{\fa_i\}_i$ of ideals such that $(X,\fa_i^q)$ is canonical of semistable type, we want to bound the log discrepancy $a_{E_i}(X)$ of some divisor $E_i$ that computes $\mld_x(X,\fa_i^q\fm^s)$. For each $i$, we take a crepant weighted blow-up $F_i\subset B_i\to x\in X$ maximally, with respect to weights $(w_{1i},w_{2i},1)$ for $w_{2i}\le w_{1i}$. The substantial case is when $w_{2i}$ diverges and the slope $w_{1i}/w_{2i}$ converges to a rational number $\mu=w_1/w_2$ in a finite set $\sfQ_n$. We bound the log discrepancy $a_{E_i}(B_i)$ in this case (Lemma~\ref{lem:slopeQ}). Indeed, we can assume the existence of a crepant weighted blow-up of $B_i$ at the centre of $E_i$ with weights $(v_{1i},v_{2i},1)$ such that $v_{1i}/v_{2i}$ converges to a limit $\mu'\in\sfQ_n$. By the maximal choice of $B_i$, it follows from the result on composition of weighted blow-ups that $\mu<\mu'$, and we obtain the bound by induction on $\mu$.

The divisor $F_i$ tends to the divisor $F$ obtained by the weighted blow-up of $X$ with weights $(w_1,w_2)$. Since $a_{E_i}(B_i)$ is bounded, we may truncate $\fa_i$ by the part of higher order along $F$. The truncated general member admits a flat degeneration to a weighted homogeneous function $h_i$ such that $a_F(X,h_i^q)$ is zero. By the lower semi-continuity of minimal log discrepancies on smooth varieties \cite{EMY03}, the boundedness for $(X,\fa_i^q\fm^s)$ is reduced to that for the degenerate pairs $(X,h_i^q\fm^s)$. By precise inversion of adjunction, the latter is derived from the result on the surface $F$ (Section~\ref{sct:monomial}).

\section{Preliminaries}
We shall fix notation and recall basic definitions following our preceding paper \cite{K21}. We work over an algebraically closed field $k$ of characteristic zero. An \textit{algebraic scheme} is a separated scheme of finite type over $\Spec k$. It is a \textit{variety} if it is integral. An ideal sheaf on an algebraic scheme is assumed to be coherent. The germ is considered at a closed point. The natural numbers begin with zero. We set $\bN_+=\bN\setminus\{0\}$.

Let $X$ be an algebraic scheme and let $Z$ be a closed subvariety of $X$. The \textit{order} $\ord_Z\fa$ along $Z$ of an ideal sheaf $\fa$ in $\sO_X$ is the supremum of the integers $\nu$ such that $\fa\sO_{X,\eta}\subset\fp^\nu\sO_{X,\eta}$ for the ideal sheaf $\fp$ in $\sO_X$ defining $Z$ and the generic point $\eta$ of $Z$. For a prime divisor $E$ on a variety $Y$ equipped with a birational morphism to $X$, we write $\ord_E\fa$ for $\ord_E\fa\sO_Y$. For a function $f$ in $\sO_X$, we write $\ord_Zf$ for $\ord_Z(f\sO_X)$. When $X$ is normal, the order $\ord_ZD$ of an effective $\bQ$-Cartier divisor $D$ on $X$ is defined as $r^{-1}\ord_Z\sO_X(-rD)$ by a positive integer $r$ such that $rD$ is Cartier. The notion of $\ord_ZD$ is linearly extended to $\bR$-Cartier $\bR$-divisors.

An $\bR$-\textit{ideal} on an algebraic scheme $X$ is a formal product $\fa=\prod_j\fa_j^{r_j}$ of finitely many ideal sheaves $\fa_j$ in $\sO_X$ with positive real exponents $r_j$. We write $\fa^t=\prod_j\fa_j^{tr_j}$ for a positive real number $t$. The order $\ord_Z\fa$ is defined as $\sum_jr_j\ord_Z\fa_j$. The pull-back of $\fa$ by a morphism $Y\to X$ is $\fa\sO_Y=\prod_j(\fa_j\sO_Y)^{r_j}$. For a subset $I$ of the positive real numbers, we write $\fa\in I$ if all exponents $r_j$ belong to $I$. We say that $\fa$ is invertible if all $\fa_j$ are invertible, and in this case, $\fa$ defines the $\bR$-Cartier $\bR$-divisor $A=\sum_jr_jA_j$ by $\fa_j=\sO_X(-A_j)$. When we work on the germ $x\in X$, we say that the $\bR$-divisor $\sum_jr_j(f_j)$ for general $f_j\in\fa_j$ is defined by a general member of $\fa$. The $\bR$-ideal $\fa$ is said to be $\fm$-primary if all $\fa_j$ are $\fm$-primary, where $\fm$ is the maximal ideal in $\sO_X$ defining $x$.

A \textit{subtriple} $(X,\Delta,\fa)$ consists of a normal variety $X$, an $\bR$-divisor $\Delta$ on $X$ and an $\bR$-ideal $\fa$ on $X$ such that $K_X+\Delta$ is $\bR$-Cartier. It is called a \textit{triple} if $\Delta$ is effective. If $\fa=\sO_X$ or $\Delta=0$, then we simply write $(X,\Delta)$ or $(X,\fa)$ and call it a \textit{\bracketsup{sub}pair}. A divisor \textit{over} $X$ is a prime divisor $E$ on some normal variety $Y$ equipped with a birational morphism $\pi\colon Y\to X$. The closure of the image $\pi(E)$ is called the \textit{centre} of $E$ on $X$ and denoted by $c_X(E)$. Two divisors over $X$ are usually identified if they define the same valuation on the function field of $X$. The \textit{log discrepancy} of $E$ with respect to the subtriple $(X,\Delta,\fa)$ is
\[
a_E(X,\Delta,\fa)=1+\ord_EK_{Y/(X,\Delta)}-\ord_E\fa,
\]
where $K_{Y/(X,\Delta)}=K_Y-\pi^*(K_X+\Delta)$.

Let $\eta$ be a scheme-theoretic point in $X$ with closure $Z=\overline{\{\eta\}}$ in $X$. The \textit{minimal log discrepancy} of $(X,\Delta,\fa)$ at $\eta$ is
\[
\mld_\eta(X,\Delta,\fa)=\inf\{a_E(X,\Delta,\fa)\mid\textrm{$E$ a divisor over $X$},\ c_X(E)=Z\}.
\]
It is a non-negative real number or minus infinity. We say that a divisor $E$ over $X$ \textit{computes} $\mld_\eta(X,\Delta,\fa)$ if $c_X(E)=Z$ and $a_E(X,\Delta,\fa)=\mld_\eta(X,\Delta,\fa)$ (or is negative when $\mld_\eta(X,\Delta,\fa)=-\infty$). We also define the \textit{minimal log discrepancy} $\mld_W(X,\Delta,\fa)$ of $(X,\Delta,\fa)$ in a closed subset $W$ of $X$ as the infimum of $a_E(X,\Delta,\fa)$ for all divisors $E$ over $X$ such that $c_X(E)\subset W$.

The subtriple $(X,\Delta,\fa)$ is said to be \textit{log canonical} (\textit{lc}) (resp.\ \textit{Kawamata log terminal} (\textit{klt})) if $a_E(X,\Delta,\fa)\ge0$ (resp.\ $>0$) for all divisors $E$ over $X$. It is said to be \textit{purely log terminal} (\textit{plt}) (resp.\ \textit{canonical}, \textit{terminal}) if $a_E(X,\Delta,\fa)>0$ (resp.\ $\ge1$, $>1$) for all divisors $E$ exceptional over $X$. When $(X,\Delta,\fa)$ is lc, the centre $c_X(E)$ of a divisor $E$ over $X$ such that $a_E(X,\Delta,\fa)=0$ is called an \textit{lc centre}. On the germ of a variety, an lc centre contained in all lc centres is called the \textit{smallest lc centre}. The \textit{index} of a normal $\bQ$-Gorenstein singularity $x\in X$ is the least positive integer $r$ such that $rK_X$ is Cartier.

A \textit{contraction} is a projective morphism of normal varieties with connected fibres. A \textit{log resolution} of a subtriple $(X,\Delta,\fa)$ is a birational contraction $X'\to X$ from a smooth variety such that the exceptional locus is a divisor, such that $\fa\sO_{X'}$ is invertible and such that the union of the exceptional locus, the support of the strict transform of $\Delta$ and the support of the $\bR$-divisor defined by $\fa\sO_{X'}$ is simple normal crossing. Let $\pi\colon Y\to X$ be a birational morphism of normal varieties. A subtriple $(Y,\Gamma,\fb)$ is said to be \textit{crepant} to $(X,\Delta,\fa)$ if $a_E(Y,\Gamma,\fb)=a_E(X,\Delta,\fa)$ for all divisors $E$ over $Y$. Suppose that the exceptional locus of $\pi$ is a divisor $\sum_iE_i$. The \textit{weak transform} $\fa_Y=\prod_j(\fa_{jY})^{r_j}$ in $Y$ of an $\bR$-ideal $\fa=\prod_j\fa_j^{r_j}$ is defined by
\[
\fa_{jY}=\fa_j\sO_Y(\textstyle\sum_i(\ord_{E_i}\fa_j)E_i)
\]
provided that $\sum_i(\ord_{E_i}\fa_j)E_i$ is Cartier for all $j$. If the weak transform $\fc_{jY}$ in $Y$ of $\fa_j^n$ is defined for all $j$ for a common positive integer $n$, then we call $\fa_Y=\prod_j(\fc_{jY})^{r_j/n}$ a \textit{weak $\bQ$-transform} in $Y$ of $\fa$.

Let $x\in X$ be the germ of a smooth variety. Let $x_1,\ldots,x_c$ be a part of a regular system of parameters in $\sO_{X,x}$ and let $w_1,\ldots,w_c$ be positive integers. The \textit{weighted blow-up} of $X$ with $\wt(x_1,\ldots,x_c)=(w_1,\ldots,w_c)$ is $B=\Proj_X(\bigoplus_{w\in\bN}\sI_w)\to X$ for the ideal $\sI_w$ in $\sO_X$ generated by all monomials $x_1^{s_1}\cdots x_c^{s_c}$ such that $\sum_iw_is_i\ge w$. The description is reduced to the case when $x\in X$ is the germ $o\in\bA^d$ at origin of the affine space with coordinates $x_1,\ldots,x_d$. Suppose that $w_1,\ldots,w_d$ have no common factors. Then the weighted blow-up of $\bA^d$ with $\wt(x_1,\ldots,x_d)=(w_1,\ldots,w_d)$ is covered by the affine charts $U_i=(x_i\neq0)\simeq\bA^d/\bZ_{w_i}(w_1,\ldots,w_{i-1},-1,w_{i+1},\ldots,w_d)$ for $1\le i\le d$.

The notation $\bA^d/\bZ_r(a_1,\ldots,a_d)$ stands for the quotient of $\bA^d$ by the cyclic group $\bZ_r$ whose generator sends $x_i$ to $\zeta^{a_i}x_i$ for a primitive $r$-th root $\zeta$ of unity. We call $x_1,\ldots,x_d$ the \textit{orbifold coordinates} on the quotient. A \textit{cyclic quotient singularity} is a singularity \'etale equivalent to some $o\in A=\bA^d/\bZ_r(a_1,\ldots,a_d)$ and it is said to be of \textit{type} $\frac{1}{r}(a_1,\ldots,a_d)$. The quotient $A$ is the toric variety $T_N(\Delta)$ for the lattice $N=\bZ^d+\bZ v$ with $v=\frac{1}{r}(a_1,\ldots,a_d)$ and the standard fan $\Delta$. For $e=\frac{1}{r}(w_1,\ldots,w_d)\in N$ with all $w_i$ non-negative, we define the \textit{weighted blow-up} of $A$ with $\wt(x_1,\ldots,x_d)=\frac{1}{r}(w_1,\ldots,w_d)$ by adding the ray generated by $e$. The definition makes sense for any cyclic quotient singularity. See \cite[6.38]{KSC04} for details.

\begin{remark}\label{rmk:wbu}
If $y_1,\ldots,y_c$ form a part of a regular system of parameters in $\sO_{X,x}$ for $y_i\in\sI_{w_i}\setminus\sI_{w_i+1}$, then $B$ is the same as the weighted blow-up of $X$ with $\wt(y_1,\ldots,y_c)=(w_1,\ldots,w_c)$.
\end{remark}

We introduce a notation peculiar to this paper.

\begin{notation}\label{ntn:PQ}
We define the set
\[
\sfP=\{(w_1,w_2)\in(\bN_+)^2\mid\textup{$w_1,w_2$ coprime},\ w_2\le w_1\}
\]
and identify $\sfP$ with the set $\sfQ=\{q\in\bQ\mid q\ge1\}$ by the \textit{slope} function $\mu\colon\sfP\to\sfQ$ which sends $(w_1,w_2)$ to $w_1/w_2$. For a positive integer $n$, we define the finite subsets $\sfP_n=\{(w_1,w_2)\in\sfP\mid w_1\le n\}$ and $\sfQ_n=\mu(\sfP_n)$.
\end{notation}

As mentioned at the beginning of the introduction, the four conjectures for minimal log discrepancies are equivalent when the variety is fixed. This was pointed out in \cite{MN18} and proved in \cite[theorem~4.6]{K21}.

\begin{theorem}\label{thm:equiv}
Fix the germ $x\in X$ of a klt variety. The following are equivalent.
\begin{enumerate}
\item
Fix a subset $I$ of the positive real numbers which satisfies the DCC\@. Then the set of $\mld_x(X,\fa)$ for $\bR$-ideals $\fa$ with $\fa\in I$ satisfies the ACC\@.
\item
Fix a non-negative real number $a$ and a subset $I$ of the positive real numbers which satisfies the DCC\@. Then the set of non-negative real numbers $t$ such that $\mld_x(X,\fa\fb^t)=a$ for $\bR$-ideals $\fa$ and $\fb$ with $\fa,\fb\in I$ satisfies the ACC\@.
\item
Fix a finite subset $I$ of the positive real numbers. Then there exists a positive integer $l$ depending only on $X$ and $I$ such that for $\bR$-ideals $\fa=\prod_{j=1}^e\fa_j^{r_j}$ and $\fb=\prod_{j=1}^e\fb_j^{r_j}$ on $X$, if $r_j\in I$ and $\fa_j+\fm^l=\fb_j+\fm^l$ for all $j$, where $\fm$ is the maximal ideal in $\sO_X$ defining $x$, then $\mld_x(X,\fa)=\mld_x(X,\fb)$.
\item
Fix a finite subset $I$ of the positive real numbers. Then there exists a positive integer $l$ depending only on $X$ and $I$ such that for an $\bR$-ideal $\fa$ on $X$, if $\fa\in I$, then there exists a divisor $E$ over $X$ which computes $\mld_x(X,\fa)$ and has $a_E(X)\le l$.
\end{enumerate}
\end{theorem}

Following \cite{HLLaxv}, one can extend the third and fourth statements to the case when $I$ is a subset which satisfies the DCC as in Proposition~\ref{prp:HLL}.

\begin{remark}\label{rmk:equiv}
We collect some known results on the four statements above.
\begin{enumerate}
\item\label{itm:equiv-sf}
The statements hold when $X$ is a klt surface \cite{Al93}, \cite{K13}, \cite[theorem~1.3]{MN18}.
\item
The first statement, the ACC for minimal log discrepancies, holds when $I$ is finite \cite{K14}. In this case, the set of the minimal log discrepancies is finite.
\item\label{itm:equiv-nonpos}
The fourth statement, Nakamura's boundedness, holds for $\bR$-ideals $\fa$ such that $\mld_x(X,\fa)$ is not positive \cite[theorem~4.8]{K21}.
\end{enumerate}
\end{remark}

The last item in the remark is a consequence of the $\fm$-adic semi-continuity of lc thresholds, stated as Theorem~\ref{thm:dFEM} in terms of the generic limit of ideals, due to de Fernex, Ein and Musta\c{t}\u{a}. They applied it to the first proof of the ACC for lc thresholds on smooth varieties. Later we shall need the following application of it.

\begin{theorem}[{\cite[proposition~4.12, corollary~4.13]{K21}}, \cite{MN18}]\label{thm:lct}
Let $x\in X$ be the germ of a klt variety and let $\fm$ be the maximal ideal in $\sO_X$ defining $x$. Fix a finite subset $I$ of the positive real numbers. Then there exist a positive real number $t$ and a positive integer $b$ depending only on $X$ and $I$ such that for an $\bR$-ideal $\fa$ on $X$, if $\fa\in I$ and $\mld_x(X,\fa)$ is positive, then $(X,\fa\fm^t)$ is lc and $\ord_E\fm\le b$ for every divisor $E$ over $X$ that computes $\mld_x(X,\fa)$.
\end{theorem}

One of the main results in \cite{K21} is the reduction of the theorems in the introduction to the case when the boundary splits into a canonical part and the maximal ideal to some power. Theorems~\ref{thm:main} to~\ref{thm:nakamura} follow from the next theorem with the aid of Proposition~\ref{prp:HLL} as stated in \cite[theorem~1.1]{K21} and we shall prove it in the present paper.

\begin{theorem}\label{thm:product}
Let $x\in X$ be the germ of a smooth threefold and let $\fm$ be the maximal ideal in $\sO_X$ defining $x$. Fix a positive rational number $q$ and a non-negative rational number $s$. Then there exists a positive integer $l$ depending only on $q$ and $s$ such that for an ideal $\fa$ in $\sO_X$, if $(X,\fa^q)$ is canonical, then there exists a divisor $E$ over $X$ which computes $\mld_x(X,\fa^q\fm^s)$ and has $a_E(X)\le l$.
\end{theorem}

\begin{remark}\label{rmk:product}
Our prior work \cite[theorems~1.4,~1.5]{K21} settled the theorem in the following cases.
\begin{enumerate}
\item\label{itm:product-term}
$(X,\fa^q)$ is terminal.
\item
$s$ is zero.
\item\label{itm:product-half}
$\mld_x(X,\fa^q\fm^{1/2})$ is not positive.
\item\label{itm:product-one}
$(X,\fa^q\fm)$ is lc.
\end{enumerate}
\end{remark}

One can assume the ideal $\fa$ to be $\fm$-primary by adding $\fm$ to high power $n$ such that $\mld_x(X,\fa^q\fm^s)=\mld_x(X,(\fa+\fm^n)^q\fm^s)$. This allows the following reduction.

\begin{remark}\label{rmk:etale}
In the study of Theorem~\ref{thm:product} or a statement of the same kind, one can assume $x\in X$ to be the germ $o\in\bA^3$ at origin of the affine space as mentioned in \cite[remark~7.1]{K21}. Indeed, there exists an \'etale morphism from $x\in X$ to $o\in\bA^3$ and every $\fm$-primary ideal $\fa$ in $\sO_X$ is the pull-back of an $\fn$-primary ideal $\fb$ in $\sO_{\bA^3}$ with $\mld_x(X,\fa^q\fm^s)=\mld_o(\bA^3,\fb^q\fn^s)$, where $\fn$ is the maximal ideal in $\sO_{\bA^3}$ defining $o$. As explained at the end of section~2 in \cite{K21}, for an $\bR$-ideal $\fa$ on $X$, $\mld_x(X,\fa)$ equals the minimal log discrepancy $\mld_{\hat x}(\hat X,\fa\sO_{\hat X})$ defined on the spectrum $\hat x\in\hat X$ of the completion of the local ring $\sO_{X,x}$.
\end{remark}

The main ingredient for the reduction to Theorem~\ref{thm:product} is a classification of threefold divisorial contractions quoted below, due to Kawamata and the author. A \textit{divisorial contraction} means a birational contraction between $\bQ$-factorial terminal varieties such that the exceptional locus is a prime divisor. The reader may refer to chapters~3 and~4 in the book \cite{K24} for a systematic classification. It contains a simplified proof \cite[theorem~6.2.5]{K24} of Stepanov's ACC for $1$-lc thresholds on smooth threefolds \cite{St11}, which is an application of Theorem~\ref{thm:divcont}(\ref{itm:divcont-sm}) and used in the reduction. Recall that every quotient terminal threefold singularity is cyclic and it is of type $\frac{1}{r}(w,-w,1)$ with coprime integers $r$ and $w$.

\begin{theorem}\label{thm:divcont}
Let $\pi\colon E\subset Y\to x\in X$ be a threefold divisorial contraction which contracts the divisor $E$ to a closed point $x$.
\begin{enumerate}
\item\label{itm:divcont-quot}
\textup{(\cite{Ka96})}\;
If $x\in X$ is a quotient singularity of type $\frac{1}{r}(w,-w,1)$ with orbifold coordinates $x_1,x_2,x_3$ for $0<w<r$, then $\pi$ is the weighted blow-up with $\wt(x_1,x_2,x_3)=\frac{1}{r}(w,r-w,1)$.
\item\label{itm:divcont-sm}
\textup{(\cite{K01})}\;
If $x\in X$ is a smooth point, then $\pi$ is the weighted blow-up with $\wt(x_1,x_2,x_3)=(w_1,w_2,1)$ for some regular system $x_1,x_2,x_3$ of parameters in $\sO_{X,x}$ and some $(w_1,w_2)\in\sfP$ in Notation~\textup{\ref{ntn:PQ}}.
\end{enumerate}
\end{theorem}

The threefold $Y$ in the theorem has only quotient singularities. In (\ref{itm:divcont-quot}) it has two singularities of types $\frac{1}{w}(-r,r,1)$ and $\frac{1}{r-w}(r,-r,1)$, whilst in (\ref{itm:divcont-sm}) it has two singularities of types $\frac{1}{w_1}(-1,w_2,1)$ and $\frac{1}{w_2}(w_1,-1,1)$.

The object in Theorem~\ref{thm:product} is a pair $(X,\fa^q\fm^s)$ such that $(X,\fa^q)$ is canonical. We may assume that $\fa$ is $\fm$-primary as remarked. By virtue of Remark~\ref{rmk:product}(\ref{itm:product-term})(\ref{itm:product-half}), the theorem has been settled unless $\mld_x(X,\fa^q)=1$ and $\mld_x(X,\fa^q\fm^{1/2})>0$. In the remaining case, every divisor $E$ over $X$ computing $\mld_x(X,\fa^q)=1$ has $\ord_E\fm=1$ since $a_E(X,\fa^q\fm^{1/2})=1-(1/2)\ord_E\fm>0$. For a canonical surface singularity, this amounts to the case when it is Du Val of type $A$. Inspired by the philosophy in \cite{KM92}, we make the following definition.

\begin{definition}\label{dfn:semistable}
Let $x\in X$ be the germ of a smooth threefold and let $\fa$ be an $\bR$-ideal on $X$. We say that the pair $(x\in X,\fa)$ is canonical \textit{of semistable type} if $\mld_x(X,\fa)=1$ and in addition $\ord_E\fm=1$ for every divisor $E$ over $X$ that computes $\mld_x(X,\fa)$, where $\fm$ is the maximal ideal in $\sO_X$ defining $x$.
\end{definition}

\section{Composition of weighted blow-ups}
Let $E\subset Y\to x\in X$ be the weighted blow-up of the germ of a smooth threefold with $\wt(x_1,x_2,x_3)=(w_1,w_2,1)$, which contracts the divisor $E$ to the point $x$, for a pair $(w_1,w_2)$ of positive integers such that $w_2\le w_1$. By Theorem~\ref{thm:divcont}(\ref{itm:divcont-sm}), it is a divisorial contraction if and only if $(w_1,w_2)$ belongs to the set $\sfP$ in Notation~\ref{ntn:PQ}. If $(w_1,w_2)\not\in\sfP$, then $w_1$ and $w_2$ have a common factor and $Y$ is singular along the intersection of $E$ with the strict transform of the divisor on $X$ defined by $x_3$. In this section, we shall discuss the problem of when the divisor obtained by a composition of weighted blow-ups is still obtained by a weighted blow-up.

\begin{theorem}\label{thm:composite}
Let $x\in X$ be the germ of a smooth threefold and let $\fa$ be an $\bR$-ideal on $X$ such that $\mld_x(X,\fa)=1$. Let $E\subset Y\to x\in X$ be the weighted blow-up with $\wt(x_1,x_2,x_3)=(w_1,w_2,1)$ for $w_2\le w_1$ such that $a_E(X,\fa)=1$. Let $y\in Y$ be the point above $x$ in the strict transform of the curve in $X$ defined by $(x_1,x_2)\sO_X$. Let $F\subset Z\to y\in Y$ be the weighted blow-up with $\wt(y_1,y_2,x_3)=(v_1,v_2,1)$ for $v_2\le v_1$ with respect to a regular system $y_1,y_2,x_3$ of parameters in $\sO_{Y,y}$ such that $a_F(Y,\fa_Y)=1$ for the weak transform $\fa_Y$ in the germ $y\in Y$ of $\fa$.

If $\rd{(v_1-1)/v_2}\le(w_1-v_1^2)/w_2$, then the divisor $F$ over $X$ is obtained by the weighted blow-up of $X$ with $\wt(x_1',x_2',x_3)=(w_1+v_1-d,w_2+v_2+d,1)$ for some regular system $x_1',x_2',x_3$ of parameters in $\sO_{X,x}$ and some integer $0\le d\le v_1-v_2$.
\end{theorem}

\begin{proof}
We write $\varphi\colon Z\to Y$. First of all, we remark the estimate $w_2+v_1\le w_1+v_2$. Unless $v_1=v_2$, one has $1\le\rd{(v_1-1)/v_2}\le(w_1-v_1^2)/w_2$, from which $w_2+v_1\le(w_1-v_1^2)+v_1<w_1+v_2$.

\textit{Step}~1.
We have a regular system $(z_1,z_2,x_3)=(x_1x_3^{-w_1},x_2x_3^{-w_2},x_3)$ of parameters in $\sO_{Y,y}$. For $v\in\bN$, the ideal sheaf $\fn(v)=\varphi_*\sO(-vF)$ in $\sO_Y$ consists of the functions with order at least $v$ along $F$. The $k$-vector space $V_v=\fn(v)/\fn(v+1)$ has a monomial basis consisting of all $y_1^{s_1}y_2^{s_2}x_3^{s_3}$ such that $v_1s_1+v_2s_2+s_3=v$.

If $v<v_2$, then $V_v$ has a basis $x_3^v$. Hence the function $z_i\in\sO_Y$ for $i=1,2$ is congruent modulo $\fn(v_2)$ to some polynomial $p_i(x_3)$ in $x_3$. Then $x_i'=x_i-x_3^{w_i}p_i$ satisfies $x_i'x_3^{-w_i}=z_i-p_i\in\fn(v_2)$. Replacing $x_i$ by $x_i'$, we may and shall assume that $\ord_Fz_i\ge v_2$ or equivalently $\ord_Fx_i\ge w_i+v_2$.

If $v_2=v_1$, then $\ord_Fx_i\ge w_i+v_i$ for $i=1,2$ and $\ord_Fx_3=1$. We apply Proposition~\ref{prp:parallel} to the weighted blow-up $F'\subset B\to x\in X$ with $\wt(x_1,x_2,x_3)=(\ord_Fx_1,\ord_Fx_2,1)$. Because the log discrepancy $a_{F'}(X)=\ord_Fx_1+\ord_Fx_2+1$ of the exceptional divisor $F'$ is not less than $a_F(X)=w_1+v_1+w_2+v_2+1$, it follows that $F=F'$ as a divisor over $X$ with $\ord_Fx_i=w_i+v_i$ for $i=1,2$. Thus we shall assume that $v_2<v_1$.

\textit{Step}~2.
Since $v_2<v_1$, the vector space $V_{v_2}$, spanned by $x_1,x_2,x_3^{v_2}$, is of dimension two. We shall use the notation $I_v=\{(a,b)\in\bN^2\mid v_2a+b=v\}$.

Suppose that $z_2$ and $x_3^{v_2}$ form a basis of $V_{v_2}$. In this case, by Remark~\ref{rmk:wbu}, $\varphi$ is the same as the weighted blow-up with $\wt(y_1,z_2,x_3)=(v_1,v_2,1)$. If $v<v_1$, then $V_v$ has a basis $\{z_2^ax_3^b\}_{(a,b)\in I_v}$. Hence $z_1$ is congruent modulo $\fn(v_1)$ to some polynomial $q_1(z_2,x_3)=\sum_{v=v_2}^{v_1-1}\sum_{(a,b)\in I_v}c_{ab}z_2^ax_3^b$ for $c_{ab}\in k$. In the expression
\[
z_1-q_1=(x_1-\textstyle\sum c_{ab}x_2^ax_3^{w_1-w_2a+b})x_3^{-w_1},
\]
the exponent $w_1-w_2a+b$ is positive by $a\le\rd{(v_1-1)/v_2}<w_1/w_2$. Replacing $x_1$ by $x_1-\sum c_{ab}x_2^ax_3^{w_1-w_2a+b}$, one has $\ord_Fz_1\ge v_1$ and hence $\ord_Fx_1\ge w_1+v_1$ as well as $\ord_Fx_2\ge w_2+v_2$. Then by Proposition~\ref{prp:parallel}, $F$ is obtained by the weighted blow-up of $X$ with $\wt(x_1,x_2,x_3)=(w_1+v_1,w_2+v_2,1)$.

\textit{Step}~3.
Henceforth we assume, in addition to $v_2<v_1$, that $z_2$ and $x_3^{v_2}$ are linearly dependent in $V_{v_2}$, that is, $z_2$ is congruent to $cx_3^{v_2}$ modulo $\fn(v_2+1)$ for some $c\in k$. Replacing $x_2$ by $x_2-cx_3^{w_2+v_2}$, we assume that $\ord_Fz_2>v_2$. In this case, $\varphi$ is the same as the weighted blow-up with $\wt(y_1,z_1,x_3)=(v_1,v_2,1)$. If $v<v_1$, then $V_v$ has a basis $\{z_1^ax_3^b\}_{(a,b)\in I_v}$. Hence $z_2$ is congruent modulo $\fn(v_1)$ to some $q_2(z_1,x_3)=\sum_{v=v_2+1}^{v_1-1}\sum_{(a,b)\in I_v}c_{ab}z_1^ax_3^b$. Replacing $x_2$ by $x_2-x_3^{w_2}\sum c_{0b}x_3^b$, we may assume that $c_{0b}=0$ for all $b$.

If $q_2=0$, then $\ord_Fz_2\ge v_1$ and Proposition~\ref{prp:parallel} shows that $F$ is obtained by the weighted blow-up of $X$ with $\wt(x_1,x_2,x_3)=(w_1+v_2,w_2+v_1,1)$. We shall assume that $q_2\neq0$. Let $v_2'$ be the minimum of $v$ such that $c_{ab}\neq0$ for some $(a,b)\in I_v$. Let $a'$ be the maximum of $a$ such that $c_{a,v_2'-v_2a}\neq0$. Note that $v_2<v_2'<v_1$ and $1\le a'\le\rd{v_2'/v_2}$.

Consider the weighted blow-up $G\subset W\to x\in X$ with $\wt(x_1,x_2,x_3)=(w_1+v_2,w_2+{v_2'},1)$. The exceptional divisor $G$ is identified with the weighted projective space $\bP(w_1+v_2,w_2+v_2',1)$ with weighted homogeneous coordinates $x_1,x_2,x_3$. Because $q_2'=z_2-\sum_{(a,b)\in I_{v_2'}}c_{ab}z_1^ax_3^b$ belongs to $\fn(v_2'+1)$, the centre $c_W(F)$ of $F$ is contained in the curve $C$ in $G$ defined by the weighted homogeneous polynomial
\[
x_2x_3^{(w_1+v_2)a'-w_2-v_2'}-\sum_{a=1}^{a'}c_{a,v_2'-v_2a}x_1^ax_3^{(w_1+v_2)(a'-a)}=q_2'x_3^{(w_1+v_2)a'-v_2'}
\]
in $x_1,x_2,x_3$. We remark that the exponent $(w_1+v_2)a'-w_2-v_2'$ is positive from the estimate $w_2+v_1\le w_1+v_2$ at the beginning of the proof. If the centre $c_W(F)$ were a point, then it would be contained in some curve in $G$ defined by a linear combination of $x_2$ and $x_3^{w_2+v_2'}$. This means the linear dependence in $V_{v_2'}$ of $z_2$ and $x_3^{v_2'}$, which contradicts $c_{0v_2'}=0$. Thus the centre $c_W(F)$ is the curve $C$.

We claim that $a'=1$. Let $\fa_W$ be a weak $\bQ$-transform in $W$ of $\fa$. If $\ord_C\fa_W$ were less than one, then by \cite[proposition~6]{K17}, the divisor $\Gamma$ obtained at the generic point $\eta$ of $C$ by the blow-up of $W$ along $C$ should compute $\mld_\eta(W,\fa_W)$, in which $1<2-\ord_C\fa_W=a_\Gamma(W,\fa_W)=\mld_\eta(W,\fa_W)\le a_F(W,\fa_W)$. This contradicts
\[
a_F(W,\fa_W)\le a_F(W,\fa_W)+(a_G(X,\fa)-1)\ord_FG=a_F(X,\fa)=1.
\]
Hence $\ord_C\fa_W\ge1$. Since $C$ is of weighted degree $(w_1+v_2)a'$, one obtains
\[
(w_1+v_2)a'\le(w_1+v_2)a'\ord_C\fa_W\le\ord_G\fa
\]
and consequently
\begin{align*}
1\le a_G(X,\fa)=a_G(X)-\ord_G\fa&\le(w_1+v_2+w_2+v_2'+1)-(w_1+v_2)a'\\
&\le(w_1+v_2)(1-a')+w_2+v_1\le(w_1+v_2)(2-a').
\end{align*}
Thus $a'=1$ as claimed.

\textit{Step}~4.
We set $d=v_2'-v_2$ and replace $x_1$ by $c_{1d}x_1$ to make $c_{1d}=1$. Then the curve $C$ in $G$ is given by $x_{1v_2'}=x_1-x_2x_3^{w_1-w_2-d}$, where $w_1-w_2-d$ is positive as remarked. Consider $x_{1v_2'}$ to be a function in $\sO_X$. Then the function $z_{1v_2'}=x_{1v_2'}x_3^{d-w_1}=z_1x_3^d-z_2$ in $\sO_Y$ belongs to $\fn(v_2'+1)$. In particular, $\ord_Fz_2=v_2'$ and $\ord_Fx_2=w_2+v_2'$.

Starting with $x_{1v_2'}$, we shall inductively construct a function $x_{1v}$ in $\sO_X$ for $v_2'<v<v_1$ such that $x_{1v}-x_{1,v-1}$ is a polynomial in $x_2,x_3$ and such that $z_{1v}=x_{1v}x_3^{d-w_1}$ belongs to $\fn(v+1)$. Suppose that $x_{1,v-1}$ has been constructed. The function $z_{1,v-1}=x_{1,v-1}x_3^{d-w_1}\in\fn(v)$ is congruent modulo $\fn(v+1)$ to some polynomial $\sum_{(a,b)\in I_v} c_{ab}z_1^ax_3^b$, where $c_{ab}$ is not the same as that in $q_2$. Then
\[
z_{1,v-1}x_3^l\equiv\textstyle\sum c_{ab}(z_1x_3^d)^ax_3^{b-da+l}\equiv\sum c_{ab}z_2^ax_3^{b-da+l}
\]
modulo $\fn(l+v+1)$ for large $l$, where the latter congruence follows from the containment $z_1x_3^d-z_2\in\fn(v_2'+1)$. Hence
\[
x_{1v}=x_{1,v-1}-\textstyle\sum c_{ab}x_2^ax_3^{w_1-w_2a-d(a+1)+b}=(z_{1,v-1}-\sum c_{ab}z_2^ax_3^{b-da})x_3^{w_1-d}
\]
has order greater than $w_1-d+v$ along $F$. Notice that
\begin{align*}
w_1-w_2a-d(a+1)&>w_1-w_2\Bigl\lfloor\frac{v_1-1}{v_2}\Bigr\rfloor-(v_1-v_2)\Bigl(\Bigl\lfloor\frac{v_1-1}{v_2}\Bigr\rfloor+1\Bigr)\\
&\ge w_1-(w_1-v_1^2)-(v_1-v_2)\cdot\frac{v_1+v_2}{v_2}
=v_1^2\Bigl(1-\frac{1}{v_2}\Bigr)+v_2>0.
\end{align*}
Thus $x_{1v}$ is a desired function.

One has $\ord_Fx_{1,v_1-1}\ge w_1+v_1-d$ and $\ord_Fx_2=w_2+v_2+d$. By Proposition~\ref{prp:parallel}, $F$ is obtained by the weighted blow-up of $X$ with $\wt(x_{1,v_1-1},x_2,x_3)=(w_1+v_1-d,w_2+v_2+d,1)$.
\end{proof}

\begin{remark}\label{rmk:composite}
In the theorem, let $B$ be the weighted blow-up of $X$ with $\wt(x_1',x_2',x_3)=(w_1+v_1-d,w_2+v_2+d,1)$. Then the rational map $Z\dasharrow B$ over $X$ is isomorphic outside the strict transform $E_Z$ in $Z$ of $E$. Indeed, $B$ is the output of the $(K_Z+A_Z+\varepsilon E_Z)\equiv_X\varepsilon E_Z$-MMP over $X$ for the $\bR$-divisor $A_Z$ defined by a general member of a weak $\bQ$-transform in $Z$ of $\fa$ and for a small positive real number $\varepsilon$.
\end{remark}

\begin{proposition}[{\cite[proposition~3.5.3]{K24}}]\label{prp:parallel}
Let $x\in X$ be the germ of a smooth variety and let $F\subset B\to x\in X$ be the weighted blow-up with $\wt(x_1,\ldots,x_d)=(w_1,\ldots,w_d)$ which contracts the divisor $F$ to the point $x$. For a divisor $E$ over $X$ with centre $x$, if the vector $(w_1,\ldots,w_d)$ is parallel to $(\ord_Ex_1,\ldots,\ord_Ex_d)$, then the centre $c_B(E)$ is not contained in the strict transform of the divisor on $X$ defined by the product $x_1\cdots x_d$. In particular $a_F(X)\le a_E(X)$ and the equality holds if and only if $E=F$ as a divisor over $X$.
\end{proposition}

In general, the divisor obtained by a composition of weighted blow-ups is not necessarily obtained by a weighted blow-up.

\begin{example}\label{exl:composite}
Let $E\subset Y\to x\in X$ be the weighted blow-up with $\wt(x_1,x_2,x_3)=(3,2,1)$. Take the point $y\in Y$ above $x$ in the strict transform of the curve in $X$ defined by $(x_1,x_2)\sO_X$, which is endowed with a regular system $(y_1,y_2,x_3)=(x_1x_3^{-3},x_2x_3^{-2},x_3)$ of parameters in $\sO_{Y,y}$. Let $F\subset Z\to y\in Y$ be the weighted blow-up with $\wt(y_1+y_2^2,y_2,x_3)=(5,2,1)$. Then the order $\ord_Ex_1=7$ is the maximum of $\ord_Ez_1$ for all $z_1\in\fm\setminus\fm^2$, where $\fm$ is the maximal ideal in $\sO_X$ defining $x$. The order $\ord_Ex_2=4$ is the maximum of $\ord_Ez_2$ for all $z_2\in\fm\setminus(x_1\sO_X+\fm^2)$, and $\ord_E\fm=1$. Hence if $F$ were obtained by a weighted blow-up of $X$, then it would be with respect to weights $(7,4,1)$. This contradicts $a_F(X)=13$.
\end{example}

\section{Formulation by the generic limit}\label{sct:limit}
In this section, we recall the formulation by the theory of the generic limit of ideals initiated by de Fernex and Musta\c{t}\u{a} \cite{dFM09}. We follow our style in \cite{K21} and refer the reader to sections~3 and~4 therein.

Let $x\in X$ be the germ of an algebraic scheme and let $\fm$ be the maximal ideal in $\sO_X$ defining $x$. Let $\cS=\{(\fa_{i1},\ldots,\fa_{ie})\}_{i\in\bN_+}$ be an infinite sequence of $e$-tuples of ideals $\fa_{ij}$ in $\sO_X$. There exists a \textit{family} $\cF=(Z_l,(\fa_j(l))_j,N_l,s_l,t_l)_{l\ge l_0}$ \textit{of approximations} of $\cS$ \cite[definition~3.1]{K21}, where $l_0$ is a positive integer. For $1\le j\le e$, an ideal sheaf $\fa_j(l)$ in $\sO_{X\times Z_l}$ containing $\fm^l\sO_{X\times Z_l}$ defines a flat family of ideals in $\sO_X$ parametrised by a variety $Z_l$. A map $s_l\colon N_l\to Z_l(k)$ from an infinite subset of $\bN_+$ to the set of $k$-points in $Z_l$ has dense image and is compatible with a dominant morphism $t_l\colon Z_{l+1}\to Z_l$ in such a manner that
\[
\fa_j(l)\sO_{X\times Z_{l+1}}=\fa_j(l+1)+\fm^l\sO_{X\times Z_{l+1}}
\]
and $t_l\circ s_{l+1}=s_l|_{N_{l+1}}$ with $N_{l+1}\subset N_l$. The ideal $\fa_j(l)_i$ in $\sO_X$ given by the point $s_l(i)\in Z_l$ approximates $\fa_{ij}$ in the sense that $\fa_j(l)_i=\fa_{ij}+\fm^l$. In our argument, $\cF$ is often replaced by a subfamily \cite[definition~3.5]{K21}, in which $\cS$ is replaced by an infinite subsequence.

Define the field $K$ to be the direct limit $\varinjlim_lk(Z_l)$ of the function fields $k(Z_l)$ of $Z_l$ and let $\hat X$ be the spectrum of the completion of the local ring $\sO_{X,x}\otimes_kK$. The \textit{generic limit} of $\cS$ with respect to $\cF$ is the $e$-tuple $(\sfa_1,\ldots,\sfa_e)$ of ideals in $\sO_{\hat X}$ defined by the inverse limit
\[
\sfa_j=\varprojlim_l\fa_j(l)_K/(\fm^l\otimes_kK)
\]
of the ideals $\fa_j(l)_K=\fa_j(l)\otimes_{\sO_{Z_l}}K$ in $\sO_X\otimes_kK$. The $\bR$-ideal $\sfa=\prod_{j=1}^e\sfa_j^{r_j}$ on $\hat X$ is the \textit{generic limit} of the sequence of $\bR$-ideals $\fa_i=\prod_{j=1}^e\fa_{ij}^{r_j}$ on $X$.

The notions of singularities make sense on $\hat X$ as discussed in \cite{dFEM11}, \cite{dFM09}. Let $\hat x$ denote the closed point of $\hat X$ and let $\hat\fm$ denote the maximal ideal in $\sO_{\hat X}$.

\begin{lemma}\label{lem:resolution}
If $X$ is klt, then $\hat X$ is klt, and after replacing $\cF$ by a subfamily, one has
\[
\mld_{\hat x}(\hat X,\textstyle\prod_{j=1}^e(\sfa_j+\hat\fm^l)^{r_j})=\mld_x(X,\prod_{j=1}^e\fa_j(l)_i^{r_j})
\]
for all positive real numbers $r_1,\ldots,r_e$ and all $i\in N_l$ with $l\ge l_0$.
\end{lemma}

The equality in the lemma is derived at once from the existence of log resolutions. Every divisor $\hat E$ over $\hat X$ with centre $\hat x$ descends to a divisor $E_l$ over $X\times Z_l$ for large $l$, and after a replacement of $\cF$ which depends on $\hat E$, any component $E_i$ of the fibre of $E_l$ at $s_l(i)\in Z_l$ is a divisor satisfying $\ord_{\hat E}\sfa_j=\ord_{E_i}\fa_{ij}$ and $a_{\hat E}(\hat X,\sfa)=a_{E_i}(X,\fa_i)$. The lemma induces the inequality $\mld_{\hat x}(\hat X,\sfa)\ge\mld_x(X,\fa_i)$ after a replacement of $\cF$ which depends on $r_1,\ldots,r_e$. We seek the opposite one.

\begin{lemma}[{\cite[lemma~4.7]{K21}}]\label{lem:limit}
Suppose that $X$ is klt. If the equality $\mld_{\hat x}(\hat X,\sfa)=\mld_x(X,\fa_i)$ holds for all $i\in N_{l_0}$, then there exists a positive rational number $l$ such that for infinitely many indices $i$, there exists a divisor $E_i$ over $X$ which computes $\mld_x(X,\fa_i)$ and has $a_{E_i}(X)=l$.
\end{lemma}

\begin{remark}\label{rmk:limit}
By Theorem~\ref{thm:dFEM} below, the equality $\mld_{\hat x}(\hat X,\sfa)=\mld_x(X,\fa_i)$ holds if $\mld_{\hat x}(\hat X,\sfa)$ is not positive. It also holds if $(\hat X,\sfa)$ is klt \cite[theorem~5.1]{K15}. When $X$ is a smooth threefold, we have known the equality unless $(\hat X,\sfa)$ is lc with $\mld_{\hat x}(\hat X,\sfa)\le1$ and has the smallest lc centre of dimension one \cite{K15}.
\end{remark}

The most important achievement at present is the $\fm$-adic semi-continuity of lc thresholds mentioned prior to Theorem~\ref{thm:lct}.

\begin{theorem}[de Fernex--Ein--Musta\c{t}\u{a} \cite{dFEM10}, \cite{dFEM11}]\label{thm:dFEM}
Suppose that $X$ is klt. If $(\hat X,\sfa)$ is lc, then so is $(X,\fa_i)$ for all $i\in N_{l_0}$ after replacement of $\cF$.
\end{theorem}

Henceforth we assume that $x\in X$ is the germ of a smooth threefold. We fix a positive rational number $q$ and a non-negative rational number $s$. As explained before Definition~\ref{dfn:semistable}, we may assume in Theorem~\ref{thm:product} that $\fa$ is an $\fm$-primary ideal such that $(x\in X,\fa^q)$ is canonical of semistable type. Further by Remark~\ref{rmk:equiv}(\ref{itm:equiv-nonpos}), we may also assume that $\mld_x(X,\fa^q\fm^s)$ is positive. Hence Theorem~\ref{thm:product} is reduced to the next statement.

\begin{theorem}\label{thm:limit}
Let $x\in X$ be the germ of a smooth threefold as above. Let $\cS=\{\fa_i\}_{i\in\bN_+}$ be a sequence of $\fm$-primary ideals in $\sO_X$ such that $(x\in X,\fa_i^q)$ is canonical of semistable type and such that $\mld_x(X,\fa_i^q\fm^s)$ is positive. Then there exists a positive integer $l$ such that for infinitely many $i$, there exists a divisor $E_i$ over $X$ which computes $\mld_x(X,\fa_i^q\fm^s)$ and has $a_{E_i}(X)\le l$.
\end{theorem}

In a slightly general setting, let $\cS=\{\fa_i\}_{i\in\bN_+}$ be a sequence of ideals in $\sO_X$ such that $(X,\fa_i^q)$ is canonical. We construct a generic limit $\sfa$ of $\cS$ defined on $\hat x\in\hat X$. It follows from Lemma~\ref{lem:resolution} that $\mld_{\hat x}(\hat X,\sfa^q)\ge\mld_x(X,\fa_i^q)\ge1$ and in particular $(\hat X,\sfa^q)$ is lc. By Lemma~\ref{lem:limit} and Remark~\ref{rmk:limit}, the existence of the bound $l$ is known except the case when $(\hat X,\sfa^q)$ has the smallest lc centre of dimension one with $\mld_{\hat x}(\hat X,\sfa^q)\le1$. In the remaining case, $\mld_{\hat x}(\hat X,\sfa^q)$ equals one. Thus we have the following.

\begin{theorem}\label{thm:special}
Theorem~\textup{\ref{thm:limit}} holds unless $(\hat X,\sfa^q)$ is special as below for any generic limit $\sfa$ of $\cS$.
\end{theorem}

\begin{definition}\label{dfn:special}
The pair $(\hat X,\sfa^q)$ is said to be \textit{special} if it is lc with $\mld_{\hat x}(\hat X,\sfa^q)=1$ and has the smallest lc centre of dimension one.
\end{definition}

In the next section, we shall interpret the above notion more explicitly. We have another sufficient condition for the boundedness.

\begin{lemma}\label{lem:lct1}
Theorem~\textup{\ref{thm:limit}} holds if $(\hat X,\sfa^q\hat\fm)$ is lc for a generic limit $\sfa$ of $\cS$.
\end{lemma}

\begin{proof}
By Theorem~\ref{thm:dFEM}, $(X,\fa_i^q\fm)$ is lc for infinitely many $i$. Apply Remark~\ref{rmk:product}(\ref{itm:product-one}).
\end{proof}

In our prior work \cite{K21}, we provided a meta theorem which connects a statement involving the maximal ideal $\fm$ to that involving $\fm$-primary ideals. We shall use the following reformulation of it.

\begin{theorem}\label{thm:meta}
Let $x\in X$ be the germ of a smooth threefold as above. Fix a positive integer $b$ in addition to $q$ and $s$. Let $\cS=\{(\fa_i,\fb_i)\}_{i\in\bN_+}$ be a sequence of pairs of $\fm$-primary ideals in $\sO_X$ such that $(X,\fa_i^q)$ is canonical and such that $\fb_i$ contains $\fm^b$. Suppose that there exists a positive integer $l'$ such that for all $i$ and all integers $0\le b'\le b$, there exists a divisor $E_{ib'}$ over $X$ which computes $\mld_x(X,\fa_i^q\fm^{sb'})$ and has $a_{E_{ib'}}(X)\le l'$. Then there exists a positive integer $l$ such that for infinitely many $i$, there exists a divisor $E_i$ over $X$ which computes $\mld_x(X,\fa_i^q\fb_i^s)$ and has $a_{E_i}(X)\le l$.
\end{theorem}

\begin{proof}
This has been implicitly obtained in the proof of \cite[theorem~7.3]{K21}. In steps~4 and~5 there, we obtained $\mld_x(X,\fa_i^q\fm^{sb'})\le\mld_x(X,\fa_i^q\fb_i^s)$ and $\mld_{\hat x}(\hat X,\sfa^q\hat\fm^{sb'})=\mld_{\hat x}(\hat X,\sfa^q\sfb^s)$ for some $b'$, where $(\sfa,\sfb)$ is a generic limit of $\cS$ defined on $\hat x\in\hat X$. As observed in step~5, the existence of $l'$ implies that $\mld_{\hat x}(\hat X,\sfa^q\hat\fm^{sb'})=\mld_x(X,\fa_i^q\fm^{sb'})$ after replacement of the family of approximations, from which one has $\mld_{\hat x}(\hat X,\sfa^q\sfb^s)\le\mld_x(X,\fa_i^q\fb_i^s)$. Then the existence of $l$ follows from Lemma~\ref{lem:limit} and the discussion prior to that lemma.
\end{proof}

We shall need an extension of the argument in \cite[subsection~5.2]{K15}.

\begin{theorem}\label{thm:success}
Let $x\in X$, $q$, $s$, $b$ and $\cS=\{(\fa_i,\fb_i)\}_i$ be as in Theorem~\textup{\ref{thm:meta}} but not necessarily assume the existence of $l'$. Fix a regular system $x_1,x_2,x_3$ of parameters in $\sO_{X,x}$ and let $C$ denote the curve in $X$ defined by $(x_1,x_2)\sO_X$. Let $(\sfa,\sfb)$ be a generic limit of $\cS$ defined on $\hat x\in\hat X$. Suppose that $(\hat X,\sfa^q)$ is special and that $\hat C=C\times_X\hat X$ is the smallest lc centre of $(\hat X,\sfa^q)$. Consider the weighted blow-up $F\subset B\to x\in X$ with $\wt(x_1,x_2,x_3)=(w_1,w_2,1)$ and let $b\in B$ denote the point above $x$ in the strict transform $C_B$ of $C$. Suppose that for all $i$, $(b\in B,\fa_{iB}^q)$ is crepant to $(x\in X,\fa_i^q)$ for the weak transform $\fa_{iB}$ in $b\in B$ of $\fa_i$. Then after replacement of the family of approximations, for all $i$ the minimal log discrepancy $\mld_x(X,\fa_i^q\fb_i^s)$ equals $\mld_{\hat x}(\hat X,\sfa^q\sfb^s)$ or $\mld_b(B,\fa_{iB}^q(\fb_i\sO_B)^s)$.
\end{theorem}

\begin{proof}
We abuse the same symbol $b$ for the bound of $\fb_i$ and for the point in $B$.

\textit{Step}~1.
Replacing $q$ and $\fa_i$ by $q/w_1w_2$ and $\fa_i^{w_1w_2}$ respectively, we define the weak transform $\fa_{iB}$ globally in $B$. We note that $(B,\fa_{iB}^q(\fb_i\sO_B)^s)$ is crepant to $(X,\fa_i^q\fb_i^s)$. Replacing $\fa_i$ by $\fa_i^{nq}$ for a positive integer $n$ such that $nq$ is an integer, we shall assume that $q=1/n$.

It follows from $\fm^b\subset\fb_i$ that $\hat\fm^b\subset\sfb$. The generic limit $(\sfa,\sfb)$ is constructed from a family $\cF=(Z_l,(\fa(l),\fb(l)),N_l,s_l,t_l)_{l\ge l_0}$ of approximations of $\cS$. We shall assume that $b\le l_0$ by replacing $l_0$ by a greater integer. Then for all $i\in N_l$, the ideal $\fb(l)_i$ in $\sO_X$ given by the point $s_l(i)\in Z_l$ is described as $\fb(l)_i=\fb_i+\fm^l=\fb_i$. By Remark~\ref{rmk:limit}, we may assume that $\mld_{\hat x}(\hat X,\sfa^q\sfb^s)$ is positive, in which $\hat C$ is still the smallest lc centre of $(\hat X,\sfa^q\sfb^s)$. By Theorem~\ref{thm:dFEM}, $(X,\fa_i^q\fb_i^s)$ is lc for all $i$ after replacement of $\cF$.

Replacing $\cF$, we shall prove the inequality
\[
\mld_{\hat x}(\hat X,\sfa^q\sfb^s)\le a_E(X,\fa_i^q\fb_i^s)
\]
for all $i\in N_{l_0}$ and all divisors $E$ over $X$ that have $c_X(E)=x$ and $c_B(E)\neq b$. This implies the theorem. Indeed, for any partition $\bN_+=N_1\sqcup N_2$, $\cF$ is by definition a family of approximations of the subsequence of $\cS$ indexed by $N_1$ or by $N_2$. Thus we may assume that $\mld_x(X,\fa_i^q\fb_i^s)<\mld_b(B,\fa_{iB}^q(\fb_i\sO_B)^s)$ for all $i$, in which $\mld_x(X,\fa_i^q\fb_i^s)$ is computed by some divisor $E_i$ such that $c_B(E_i)\neq b$. Then $\mld_{\hat x}(\hat X,\sfa^q\sfb^s)\le a_{E_i}(X,\fa_i^q\fb_i^s)=\mld_x(X,\fa_i^q\fb_i^s)$ from the above inequality for $E_i$ and it is actually an equality for all $i$ after replacement of $\cF$ as explained prior to Lemma~\ref{lem:limit}.

\textit{Step}~2.
Let $\hat F\subset\hat B=B\times_X\hat X\to\hat x\in\hat X$ denote the base change of the weighted blow-up $F\subset B\to x\in X$. Then $a_{\hat F}(\hat X,\sfa^q)=a_F(X,\fa_i^q)=1$ and hence $(\hat B,\sfa_B^q)$ is crepant to $(\hat X,\sfa^q)$ for the weak transform $\sfa_B$ of $\sfa$. The sum $\hat S$ of two-dimensional lc centres of $(\hat X,\sfa^q)$ has at most two prime components (it may be empty), because it contains the smallest lc centre $\hat C$. The triple $(\hat X,\hat S,\sfc^q)$ is crepant to $(\hat X,\sfa^q)$ for the ideal $\sfc=\sfa\sO_{\hat X}(n\hat S)$, and $(\hat B,\hat S_B,\sfc_B^q)$ is crepant to $(\hat B,\sfa_B^q)$ for the strict transform $\hat S_B$ of $\hat S$ and the ideal sheaf $\sfc_B=\sfa_B\sO_{\hat B}(n\hat S_B)$.

The triple $(\hat B,\hat S_B,\sfc_B^q)$ is plt outside the strict transform $\hat C_B$ of $\hat C$. In particular, $\hat S_B$ is normal outside $\hat C_B$ and there exists a log resolution $\hat Y\to\hat B$ of $(\hat B,\hat F,\sfb\sO_{\hat B})$ which is isomorphic outside $\hat F$ and which is a log resolution of $(\hat B,\hat S_B+\hat F)$ outside $\hat C_B$. Since $\mld_{\hat F}(\hat B,\sfa_B^q(\sfb\sO_{\hat B})^s)=\mld_{\hat x}(\hat X,\sfa^q\sfb^s)$ is positive, the curve $\hat C_B$ and the components of $\hat S_B$ exhaust all lc centres of $(\hat B,\sfa_B^q(\sfb\sO_{\hat B})^s)$. Hence one can fix a positive rational number $t$ such that $(\hat B,\sfa_B^q(\sfb\sO_{\hat B})^s(\sfc\sO_{\hat B})^t)$ is lc outside $\hat C_B$.

\textit{Step}~3.
The rest of the proof is very similar to the proof of \cite[theorem~5.3]{K15}. We shall write it down for the sake of completion. Let $\hat G$ be the exceptional locus of $\hat Y\to \hat X$ and express it as a sum $\sum_j\hat G_j$ of prime divisors. Let $\hat L$ be the restriction to $\hat G$ of the strict transform in $\hat Y$ of $\hat S_B$ and let $\hat M$ be the restriction to $\hat G$ of the non-trivial locus of the weak transform $\sfc_Y$ in $\hat Y$ of $\sfc$. Let $\hat N$ be the inverse image in $\hat Y$ of $\hat C_B$. For the same reason as for the inequality~(13) in \cite{K15}, we may assume that $\hat M\setminus\hat N$ is contained in the union of the divisors $\hat G_j$ such that $t\ord_{\hat G_j}\sfc\ge1$.

After replacement of $\cF$, $\hat Y\to\hat B$ is the base change of a resolution $Y_l\to B\times Z_l$ for all $l\ge l_0$. Shrinking $Z_l$, we may assume that $\hat G_j$, $\hat L$ and $\hat M$ are the base changes of flat families $G_{jl}$, $L_l$ and $M_l$ in $Y_l$ over $Z_l$. The locus $\hat N$ is the base change of the inverse image $N_l$ in $Y_l$ of $C_B\times Z_l$. We may assume that $G_l=\sum_jG_{jl}$ is a simple normal crossing divisor such that every stratum and its restriction to $L_l\setminus N_l$ are smooth over $Z_l$. We may assume that $\fb(l)\sO_{Y_l}$ is invertible. We may further assume that for each $j$, $\ord_{G_{jl,z}}\fa(l)_z$ and $\ord_{G_{jl,z}}\fb(l)_z$ are constant and less than $l$ on $Z_l$, where we express the fibre at $z\in Z_l$ by adding the subscript $z$ such as $G_{jl,z}=G_{jl}\times_{Z_l}z$.

By the maximum $c$ of $\ord_{\hat G_j}\sfc$ for all $j$, we define the $\hat\fm$-primary ideal $\sfd=\sfc+\hat\fm^c$ in $\sO_{\hat X}$. Then $\ord_{\hat G_j}\sfd=\ord_{\hat G_j}\sfc$ and $(\hat B,\sfa_B^q(\sfb\sO_{\hat B})^s(\sfd\sO_{\hat B})^t)$ is lc outside $\hat C_B$. Since $\sfd$ is defined over the function field $k(Z_l)$ of some $Z_l$, after replacement of $\cF$, $\sfd$ is the base change of an ideal sheaf $\fd_l$ in $\sO_{X\times Z_l}$ and $\ord_{G_{jl,z}}\fd_{lz}$ is constant on $Z_l$.

\textit{Step}~4.
The proof of the inequality $\mld_{\hat x}(\hat X,\sfa^q\sfb^s)\le a_E(X,\fa_i^q\fb_i^s)$ will be divided according to whether the centre $c_{Y_{lz}}(E)$ lies on $M_{lz}$ or not, where we write $z=s_l(i)$ for brevity. Notice that $c_{Y_{lz}}(E)\not\subset N_{lz}$ as $c_B(E)\neq b$.

Let $\sfa_Y$ be the weak transform in $\hat Y$ of $\sfa$, let $\fa_Y(l)$ be that in $Y_l$ of $\fa(l)$ and let $\fa_{il}'$ be that in $Y_{lz}$ of $\fa_i$. Note that the fibre $Y_{lz}$ at $z=s_l(i)$ is independent of $l$ as far as $i\in N_l$. One attains the following equalities for all $i$ and $l$ similarly to \cite[lemma~5.4]{K15}.
\begin{itemize}
\item
$\fa_Y(l)_K\sO_{\hat Y}=\sfa_Y+\sfI_l$ for some $\sfI_l\subset\sO_{\hat Y}(-(n+1)\hat G)$.
\item
$\fa_Y(l)_z\sO_{Y_{lz}}=\fa_{il}'+\sI_{il}$ for some $\sI_{il}\subset\sO_{Y_{lz}}(-(n+1)G_{lz})$.
\end{itemize}
Since $\sfa_Y$ restricted to $\hat G$ is trivial outside $\hat L\cup\hat M$, from these equalities one derives that $\fa_{il}'$ is trivial outside $L_{lz}\cup M_{lz}$ after shrinking $Z_l$. One further attains the log canonicity of $(Y_{lz},G_{lz},(\fa'_{il})^q)$ outside $M_{lz}\cup N_{lz}$ by the exactly same argument as for \cite[lemma~5.5]{K15}, which is based on noetherian induction and is substantially the same as the argument for Lemma~\ref{lem:success} below. Take the $\bQ$-divisor $\Gamma_l=\sum_j(1-a_{\hat G_j}(\hat X,\sfa^q\sfb^s))G_{jl}$ so that the subtriple $(Y_{lz},\Gamma_{lz},(\fa_{il}')^q)$ is crepant to $(X,\fa_i^q\fb_i^s)$. Notice that $\fb_i\sO_{Y_{lz}}=(\fb(l)\sO_{Y_l})_z$ is invertible.

Suppose that $c_{Y_{lz}}(E)\not\subset M_{lz}\cup N_{lz}$. In this case, we choose any divisor $G_{jl,z}$ that contains $c_{Y_{lz}}(E)$. The above log canonicity of $(Y_{lz},G_{lz},(\fa'_{il})^q)$ implies that
\[
\ord_E(G_{lz}-\Gamma_{lz})\le a_E(Y_{lz},\Gamma_{lz},(\fa_{il}')^q).
\]
The right-hand side equals $a_E(X,\fa_i^q\fb_i^s)$, whilst the left-hand side is greater than or equal to $\ord_{G_{jl,z}}(G_{lz}-\Gamma_{lz})=a_{\hat G_j}(\hat X,\sfa^q\sfb^s)$. Thus $\mld_{\hat x}(\hat X,\sfa^q\sfb^s)\le a_{\hat G_j}(\hat X,\sfa^q\sfb^s)\le a_E(X,\fa_i^q\fb_i^s)$.

Suppose that $c_{Y_{lz}}(E)\subset M_{lz}$ but $c_{Y_{lz}}(E)\not\subset N_{lz}$. In this case, $c_{Y_{lz}}(E)$ lies on some divisor $G_{jl,z}$ that has $\ord_{G_{jl,z}}\fd_{lz}=\ord_{\hat G_j}\sfd=\ord_{\hat G_j}\sfc\ge t^{-1}$. Then
\[
\mld_{\hat x}(\hat X,\sfa^q\sfb^s)\le1\le t\ord_{G_{jl,z}}\fd_{lz}\le t\ord_E\fd_{lz}\le a_E(X,\fa_i^q\fb_i^s),
\]
where the last inequality follows from the log canonicity in Lemma~\ref{lem:success}.
\end{proof}

\begin{lemma}\label{lem:success}
After replacement of $\cF$, $(B,\fa_{iB}^q(\fb_i\sO_B)^s(\fd_{lz}\sO_B)^t)$ is lc outside $C_B$ for all $i\in I_l$ with $z=s_l(i)\in Z_l$.
\end{lemma}

\begin{proof}
We write $\sfe=(\sfb\sO_{\hat B})^s(\sfd\sO_{\hat B})^t$ and $\fe_l=(\fb(l)\sO_{B\times Z_l})^s(\fd_l\sO_{B\times Z_l})^t$ for brevity. Note that $\fe_{lz}=(\fb_i\sO_B)^s(\fd_{lz}\sO_B)^t$. After replacement of $\cF$, the restriction $\hat\Delta$ to $\hat F$ of the non-trivial locus of the weak $\bQ$-transform in $\hat B$ of $\sfa\sfb\sfd$ is the base change of the non-trivial locus $\Delta_l$ of the weak $\bQ$-transform in $B\times Z_l$ of $\fa(l)\fb(l)\fd_l$, and $\Delta_l$ is flat over $Z_l$. It is obvious that $(B,\fa_{iB}^q\fe_{lz})$ is lc outside $\Delta_{lz}$.

For each curve $\hat D$ in $\hat\Delta$ and the ideal sheaf $\hat\fp$ in $\sO_{\hat B}$ defining $\hat D$, there exists a non-negative rational number $e$ such that $\mld_{\hat\eta}(\hat B,\sfa_B^q\sfe\hat\fp^e)=0$ at the generic point $\hat\eta$ of $\hat D$ and it is computed by some divisor $\hat Q$ over $\hat B$. The curve $\hat D$ comes from a component $D_l$ of $\Delta_l$, and we may assume that $\hat Q$ is the base change of a divisor $Q_l$ over $B\times Z_l$. Let $\fa_B(l)$ be the weak transform in $B\times Z_l$ of $\fa(l)$. As above, for a fixed integer $n_Q$ greater than $w_1w_2\ord_{\hat Q}\sfa_B$, one attains
\begin{itemize}
\item
$\fa_B(l)_K\sO_{\hat B}=\sfa_B+\sfJ_l$ for some $\sfJ_l\subset\sO_{\hat Y}(-n_Q\hat F)$,
\item
$\fa_B(l)_z\sO_B=\fa_{iB}+\sJ_{il}$ for some $\sJ_{il}\subset\sO_B(-n_QF)$
\end{itemize}
for all $i$ and $l$. Unlike the proof of \cite[lemma~5.5]{K15}, we do not fix $l$. After replacement of $\cF$, for the generic point $\eta_z$ of every component of $D_{lz}$ with $z\in Z_l$, a component of $Q_{lz}$ computes $\mld_{\eta_z}(B,\fa_B(l)_z^q\fe_{lz}\fp_{lz}^e)=0$ for the ideal sheaf $\fp_l$ in $\sO_{B\times Z_l}$ defining $D_l$.

By a relative version \cite[theorem~3.11]{K21} of the effective $\fm$-adic semi-continuity of lc thresholds, there exists a finite subset $\hat P_D$ of $\hat D$ such that after replacement of $\cF$, $\hat P_D$ is the base change of $P_{Dl}\subset D_l$ and $(B,\fa_{iB}^q\fe_{lz}\fp_{lz}^e)$ is lc about $D_{lz}\setminus P_{Dl,z}$ for all $i$ and $l$. Applying the same argument to the finite set $(\hat\Delta\setminus\hat C_B)\setminus\bigcup_{\hat D}(\hat D\setminus\hat P_D)$, one attains the required log canonicity of $(B,\fa_{iB}^q\fe_{lz})$ outside $C_B$.
\end{proof}

\section{Log canonical slopes}
Let $x\in X$ be the germ of a smooth threefold and let $\fm$ be the maximal ideal in $\sO_X$ defining $x$. Fix a positive rational number $q$. In this section, $\cS=\{\fa_i\}_{i\in\bN_+}$ is a sequence of ideals in $\sO_X$ such that $(X,\fa_i^q)$ is canonical. We construct a generic limit $\sfa$ of $\cS$ on $\hat x\in\hat X$. The pair $(\hat X,\sfa^q)$ is lc and has $\mld_{\hat x}(\hat X,\sfa^q)\ge1$ as noted before Theorem~\ref{thm:special}. Recall Notation~\ref{ntn:PQ}.

\begin{definition}
For $(w_1,w_2)\in\sfP$ and $\mu=w_1/w_2\in\sfQ$, we say that $(\hat X,\sfa^q)$ admits a \textit{log canonical \bracketsup{lc} slope} $\mu$ with respect to a regular system $x_1,x_2,x_3$ of parameters in $\sO_{\hat X}$ if the divisor $\hat F$ obtained by the weighted blow-up of $\hat X$ with $\wt(x_1,x_2)=(w_1,w_2)$ has $a_{\hat F}(\hat X,\sfa^q)=0$ or equivalently $\ord_{\hat F}\sfa^q=w_1+w_2$. We just say that $(\hat X,\sfa^q)$ admits an \textit{lc slope} $\mu$ if so does it with respect to some regular system of parameters in $\sO_{\hat X}$.
\end{definition}

Suppose that $(\hat X,\sfa^q)$ admits an lc slope with respect to $x_1,x_2,x_3$ as in the definition. Let $\hat F\subset\hat Y\to\hat x\in\hat X$ denote the weighted blow-up and let $\hat C$ denote the centre $c_{\hat X}(\hat F)$ of $\hat F$, namely the curve defined by $(x_1,x_2)\sO_{\hat X}$. Then for the fibre $\hat f$ of $\hat F\to\hat C$ at $\hat x$, the divisor $\hat\Gamma$ obtained at the generic point of $\hat f$ by the blow-up of $\hat Y$ along $\hat f$ has $a_{\hat\Gamma}(\hat X,\sfa^q)\le a_{\hat\Gamma}(\hat Y,\hat F)=1$. Hence the inequality $\mld_{\hat x}(\hat X,\sfa^q)\ge1$ is an equality and $(\hat X,\sfa^q)$ is special as in Definition~\ref{dfn:special} with the smallest lc centre $\hat C$. In fact, the converse is also true.

\begin{theorem}\label{thm:lcslope}
The pair $(\hat X,\sfa^q)$ is special if and only if it admits an lc slope.
\end{theorem}

\begin{proof}
We have observed the if part. The only-if part immediately follows from the application \cite[theorem~6.2]{K21} of a classification of divisors over a smooth surface computing the minimal log discrepancy.
\end{proof}

Let $n$ be a positive integer such that $nq$ is an integer. One can choose the slope $\mu$ from the finite set $\sfQ_n$.

\begin{lemma}\label{lem:lcslope}
If $(\hat X,\sfa^q)$ admits an lc slope with respect to a regular system $x_1,x_2,x_3$ of parameters in $\sO_{\hat X}$, then it also admits an lc slope belonging to $\sfQ_n$ with respect to the same system $x_1,x_2,x_3$.
\end{lemma}

\begin{proof}
Suppose that $(\hat X,\sfa^q)$ admits an lc slope $w_1/w_2$ with respect to $x_1,x_2,x_3$ for $(w_1,w_2)\in\sfP$. There exists a pair $(v_1,v_2)$ of coprime integers such that $w_1v_2-w_2v_1=1$. Since every $(v_1',v_2')\in(v_1,v_2)+\bZ(w_1,w_2)$ satisfies $w_1v_2'-w_2v_1'=1$, we may choose it so that $0\le v_1<w_1$. Suppose that $(w_1,w_2)\not\in\sfP_n$. Then $n<w_1$ and $v_1\neq0$, in which $(v_1,v_2)$ belongs to $\sfP$. It suffices to show that $\cS$ admits an lc slope $v_1/v_2$ with respect to $x_1,x_2,x_3$.

The divisor $\hat F$ obtained by the weighted blow-up of $\hat X$ with $\wt(x_1,x_2)=(w_1,w_2)$ has $a_{\hat F}(\hat X,\sfa^q)=0$. Let $\hat G$ be the divisor obtained by the weighted blow-up of $\hat X$ with $\wt(x_1,x_2)=(v_1,v_2)$. From the componentwise inequality $v_1w_1^{-1}(w_1,w_2)=(v_1,v_2-1/w_1)\le(v_1,v_2)$, one derives the inequality $v_1w_1^{-1}\ord_{\hat F}\le\ord_{\hat G}$ of order functions. Hence
\[
\ord_{\hat G}\sfa^q\ge\frac{v_1}{w_1}\ord_{\hat F}\sfa^q=\frac{v_1}{w_1}(w_1+w_2)=v_1+v_2-\frac{1}{w_1}>v_1+v_2-\frac{1}{n}.
\]
It follows that $\ord_{\hat G}\sfa^q\in n^{-1}\bZ$ satisfies $\ord_{\hat G}\sfa^q\ge v_1+v_2$, that is, $a_{\hat G}(\hat X,\sfa^q)\le0$. Since $(\hat X,\sfa^q)$ is lc, $a_{\hat G}(\hat X,\sfa^q)$ is zero.
\end{proof}

The next proposition will be used in combination with Lemma~\ref{lem:lct1}.

\begin{proposition}\label{prp:lcslopes}
If $(\hat X,\sfa^q)$ admits two different lc slopes with respect to the same regular system of parameters in $\sO_{\hat X}$, then $(\hat X,\sfa^q\hat\fm)$ is lc.
\end{proposition}

\begin{proof}
Take two different elements $(w_{1i},w_{2i})\in\sfP$ for $i=1,2$ such that $(\hat X,\sfa^q)$ admits an lc slope $w_{1i}/w_{2i}$ with respect to the same system $x_1,x_2,x_3$. Make an element $(w_{13},w_{23})=p_1(w_{11},w_{21})+p_2(w_{12},w_{22})$ in $\sfP$ by positive rational numbers $p_1$ and $p_2$. Then for the divisor $\hat F_i$ obtained by the weighted blow-up $\hat Y_i\to \hat X$ with $\wt(x_1,x_2)=(w_{1i},w_{2i})$, one has
\[
\ord_{\hat F_3}\sfa^q\ge p_1\ord_{\hat F_1}\sfa^q+p_2\ord_{\hat F_2}\sfa^q=p_1(w_{11}+w_{21})+p_2(w_{12}+w_{22})=w_{13}+w_{23},
\]
that is, $a_{\hat F_3}(\hat X,\sfa^q)\le0$. Since $(\hat X,\sfa^q)$ is lc, $a_{\hat F_3}(\hat X,\sfa^q)$ is zero.

We write $\hat Y=\hat Y_3$ and $\hat F=\hat F_3$ for brevity. The triple $(\hat Y,\hat F,\sfa_Y^q)$ is crepant to $(\hat X,\sfa^q)$ for a weak $\bQ$-transform $\sfa_Y$ of $\sfa$. For $j=1,2$, let $\hat T_j$ be the strict transform in $\hat Y$ of the divisor on $\hat X$ defined by $x_j$. The support $\hat D_j$ of $\hat T_j\cap\hat F$ is the centre $c_{\hat Y}(\hat F_j)$ of $\hat F_j$ and hence it is an lc centre of $(\hat Y,\hat F,\sfa_Y^q)$. Since $\mld_{\hat x}(\hat X,\sfa^q)=1$, the order $\ord_{\hat f}\sfa_Y$ along the fibre $\hat f$ of $\hat F\to\hat X$ at $\hat x$ is zero.

Let $\hat\Delta=(1-w_{23}^{-1})\hat D_1+(1-w_{13}^{-1})\hat D_2$ be the different on $\hat F$ defined by the adjunction $(K_{\hat Y}+\hat F)|_{\hat F}=K_{\hat F}+\hat\Delta$ as in \cite[example~2.4]{K21}. By adjunction, $\mld_{\hat f}(\hat F,\hat\Delta,(\sfa_Y\sO_{\hat F})^q)$ is at least $\mld_{\hat x}(\hat X,\sfa^q)=1$ and equals $a_{\hat f}(\hat F,\hat\Delta,(\sfa_Y\sO_{\hat F})^q)=1$. Further $\hat D_j$ is an lc centre of $(\hat F,\hat\Delta,(\sfa_Y\sO_{\hat F})^q)$. Indeed, if not, then there would exist a positive rational number $t$ such that $(\hat F,\hat\Delta+t\hat D_j,(\sfa_Y\sO_{\hat F})^q)$ is lc. From the precise inversion of adjunction in Lemma~\ref{lem:pia}(\ref{itm:pia-klt}), one deduces that $(\hat Y,\hat F+tw_{3-j,3}\hat T_j,\sfa_Y^q)$ is lc. This is absurd because $\hat D_j$ is an lc centre of $(\hat Y,\hat F,\sfa_Y^q)$.

Thus one can write $(K_{\hat Y}+\hat F+q\hat A_Y)|_{\hat F}=K_{\hat F}+\hat D_1+\hat D_2+\hat L$ by the $\bQ$-divisor $\hat A_Y$ defined by a general member of $\sfa_Y$ and an effective $\bQ$-divisor $\hat L$ on $\hat F$. Since $\ord_{\hat f}\sfa_Y$ is zero, $\hat f$ does not appear in $\hat L$. Since $(\hat L\cdot\hat f)=-((K_{\hat F}+\hat D_1+\hat D_2)\cdot\hat f)=0$, $\hat L$ is zero. Thus by Lemma~\ref{lem:pia}(\ref{itm:pia-3fold}), $\mld_{\hat x}(\hat X,\sfa^q\hat\fm)=\mld_{\hat f}(\hat F,\hat D_1+\hat D_2+\hat f)=0$, which completes the proposition.
\end{proof}

We quote the precise inversion of adjunction over complete local rings since it will be used again. For a pair $(X,S+B)$ such that $S$ is reduced and has no common components with the support of $B$, the \textit{adjunction} $\nu^*((K_X+S+B)|_S)=K_{S^\nu}+B_{S^\nu}$ holds on the normalisation $\nu\colon S^\nu\to S$ of $S$ by an effective $\bR$-divisor $B_{S^\nu}$ called the \textit{different} on $S^\nu$ of $B$.

\begin{lemma}\label{lem:pia}
\begin{enumerate}
\item\label{itm:pia-klt}
\textup{(\cite[lemma~7.6]{K21})}\;
Let $x\in X$ be the germ of a klt variety and let $S$ be a prime divisor on $X$ such that $(X,S)$ is plt. Let $\Delta$ be the different on $S$ defined by $(K_X+S)|_S=K_S+\Delta$. Let $\hat x\in\hat X$ be the spectrum of the completion of the local ring $\sO_{X,x}$ and set $\hat S=S\times_X\hat X$ and $\hat\Delta=\Delta\times_X\hat X$. Let $\sfa$ be an $\bR$-ideal on $\hat X$. If $\mld_{\hat x}(\hat X,\hat S,\sfa)\le1$, then $\mld_{\hat x}(\hat X,\hat S,\sfa)=\mld_{\hat x}(\hat S,\hat\Delta,\sfa\sO_{\hat S})$.
\item\label{itm:pia-3fold}
\textup{(\cite[proposition~7.7]{K21})}\;
Let $\hat x\in\hat X=\Spec K[[x_1,x_2,x_3]]$ over a field $K$ of characteristic zero and let $\sfa$ be an $\bR$-ideal on $\hat X$. Let $\hat Y\to\hat X$ be the weighted blow-up with $\wt(x_1,x_2)=w$ for some $w\in\sfP$, let $\hat E$ be the exceptional divisor and let $\hat f$ be the fibre of $\hat E\to\hat X$ at $\hat x$. If $a_{\hat E}(\hat X,\sfa)=0$, then $\mld_{\hat x}(\hat X,\sfa)=\mld_{\hat f}(\hat E,\hat\Delta,\sfa_Y\sO_{\hat E})$ for the different $\hat\Delta$ on $\hat E$ defined by $(K_{\hat Y}+\hat E)|_{\hat E}=K_{\hat E}+\hat\Delta$ and for a weak $\bQ$-transform $\sfa_Y$ in $\hat Y$ of $\sfa$.
\end{enumerate}
\end{lemma}

We construct a sequence of weighted blow-ups of $X$ from a weighted blow-up of $\hat X$ and vice versa.

\begin{proposition}\label{prp:down}
If $(\hat X,\sfa^q)$ admits an lc slope $w_1/w_2$ for $(w_1,w_2)\in\sfP$, then after replacement of $\cS$ by a subsequence, for all $i$ there exists a regular system $x_{1i},x_{2i},x_{3i}$ of parameters in $\sO_{X,x}$ such that $a_{F_i}(X,\fa_i^q)=1$ for the divisor $F_i$ obtained by the weighted blow-up of $X$ with $\wt(x_{1i},x_{2i},x_{3i})=(iw_1,iw_2,1)$.
\end{proposition}

\begin{proof}
Fix an index $j\in\bN_+$. We shall find, for infinitely many $i$, a regular system $x_{1i},x_{2i},x_{3i}$ of parameters in $\sO_{X,x}$ such that $a_{F_{ji}}(X,\fa_i^q)=1$ for the divisor $F_{ji}$ obtained by the weighted blow-up of $X$ with $\wt(x_{1i},x_{2i},x_{3i})=(jw_1,jw_2,1)$. Then a required subsequence can be constructed inductively.

We use the argument in step~3 of the proof of \cite[theorem~7.3]{K21}. The generic limit $\sfa$ is constructed on the spectrum of the completion of the local ring $\sO_{X,x}\otimes_kK$ from a family $\cF=(Z_l,\fa(l),N_l,s_l,t_l)_{l\ge l_0}$ of approximations of $\cS$. Let $x_1,x_2,x_3$ be the regular system of parameters in $\sO_{\hat X}$ with respect to which $(\hat X,\sfa^q)$ admits an lc slope $w_1/w_2$. Let $\hat f_j\colon\hat Y_j\to\hat X$ be the weighted blow-up with $\wt(x_1,x_2,x_3)=(jw_1,jw_2,1)$ and let $\hat F_j$ be the exceptional divisor. By Remark~\ref{rmk:wbu}, there exists a regular system $x_{1j},x_{2j},x_{3j}$ of parameters in $\sO_{X,x}\otimes_kK$ such that $\hat f_j$ is also the weighted blow-up with $\wt(x_{1j},x_{2j},x_{3j})=(jw_1,jw_2,1)$ with $x_{ij}$ regarded as elements in $\sO_{\hat X}$.

After replacement of $\cF$ by a subfamily, $\hat f_j$ descends to a projective morphism $f_{jl}\colon Y_{jl}\to X\times Z_l$ from a klt variety for all $l\ge l_0$. One can assume that $x_{1j},x_{2j},x_{3j}$ come from $\sO_{X,x}\otimes_k\sO_{Z_l}$ and that for all $i\in N_l$, their fibres $x_{1j,i},x_{2j,i},x_{3j,i}$ at $s_l(i)\in Z_l$ form a regular system of parameters in $\sO_{X,x}$. The exceptional locus of $f_{jl}$ is a $\bQ$-Cartier prime divisor $F_{jl}$. The fibre $f_{ji}\colon Y_{ji}\to X$ of $f_{jl}$ at $s_l(i)$ is the weighted blow-up of $X$ with $\wt(x_{1j,i},x_{2j,i},x_{3j,i})=(jw_1,jw_2,1)$ and $F_{ji}=F_{jl}\times_{Y_{jl}}Y_{ji}$ is the exceptional divisor of $f_{ji}$. One may assume that $\ord_{F_{ji}}\fa_i=\ord_{\hat F_j}\sfa$ for all $i\in N_l$.

Let $\hat F$ be the divisor obtained by the weighted blow-up of $\hat X$ with $\wt(x_1,x_2)=(w_1,w_2)$. It follows from the componentwise inequality $(jw_1,jw_2,1)\ge j(w_1,w_2,0)$ that $\ord_{\hat F_j}\sfa^q\ge j\ord_{\hat F}\sfa^q=j(w_1+w_2)$. Hence $a_{\hat F_j}(\hat X,\sfa^q)\le1$. Since $\mld_{\hat x}(\hat X,\sfa^q)=1$, one obtains $a_{F_{ji}}(X,\fa_i^q)=a_{\hat F_j}(\hat X,\sfa^q)=1$.
\end{proof}

\begin{proposition}\label{prp:up}
Fix a regular system $x_1,x_2,x_3$ of parameters in $\sO_{X,x}$. Suppose that for all $i$, there exists a pair $(w_{1i},w_{2i})$ of positive integers with $w_{2i}\le w_{1i}$ such that the divisor $F_i$ obtained by the weighted blow-up of $X$ with $\wt(x_1,x_2,x_3)=(w_{1i},w_{2i},1)$ has $a_{F_i}(X,\fa_i^q)=1$. Suppose that $\lim_{i\to\infty}w_{2i}=\infty$ and that $\mu_i=w_{1i}/w_{2i}\in\sfQ$ has a limit $\mu=\lim_{i\to\infty}\mu_i$ in $\bR$. Let $(w_1,w_2)$ be an element in $\sfP$. If either
\begin{enumerate}
\item
$\mu\in\sfQ$ and $\mu=w_1/w_2$, or
\item
$\mu\not\in\sfQ$ and $\abs{\mu-w_1/w_2}<\varepsilon$ for a small positive real number $\varepsilon$ depending only on $q$ and $\mu$,
\end{enumerate}
then after replacement of $\cS$ by a subsequence, $(\hat X,\sfa^q)$ admits an lc slope $w_1/w_2$ with respect to $x_1,x_2,x_3$ regarded as elements in $\sO_{\hat X}$.
\end{proposition}

\begin{proof}
Note that $\mu\in\sfQ$ if and only if $\mu\in\bQ$. Passing to a subsequence, we may assume that $\mu_i$ form a non-decreasing sequence or a non-increasing one. For each $i$, we consider the real-valued weighted order $\ord_i$ on $\sO_X$ and $\sO_{\hat X}$ with respect to $\wt(x_1,x_2,x_3)=(i\mu,i,1)$. If $\mu_i$ is non-decreasing (resp.\ non-increasing), then $\ord_i\ge iw_{2j}^{-1}\ord_{F_j}$ (resp.\ $\ord_i\ge i\mu w_{1j}^{-1}\ord_{F_j}$) as far as $i\le w_{2j}$. Hence with $\ord_{F_j}\fa_j^q=w_{1j}+w_{2j}$, one has for large $j$,
\[
\ord_i\fa_j^q\ge
\begin{cases}
\displaystyle\frac{i}{w_{2j}}(w_{1j}+w_{2j})=i(\mu+1)-i(\mu-\mu_j)&\textrm{for $\mu_i$ non-decreasing,}\\
\displaystyle\frac{i\mu}{w_{1j}}(w_{1j}+w_{2j})=i(\mu+1)-i\Bigl(1-\frac{\mu}{\mu_j}\Bigr)&\textrm{for $\mu_i$ non-increasing.}
\end{cases}
\]

Take a small positive real number $\delta_i$ such that no element $m$ in the discrete set $q(\mu\bN+\bN)$ satisfies $i(\mu+1)-\delta_i<m<i(\mu+1)$. Then as far as $i(\mu-\mu_j)<\delta_i$ (resp.\ $i(1-\mu/\mu_j)<\delta_i$), the above inequality means that $\ord_i\fa_j^q\ge i(\mu+1)$. This implies that $\ord_i\sfa^q\ge i(\mu+1)$.

Take the weighted order $\ord_\infty$ with respect to $\wt(x_1,x_2)=(\mu,1)$ and set $S=\{(s_1,s_2)\in\bN^2\mid\mu s_1+s_2=\ord_\infty\sfa\}$. We express the general member $a$ of $\sfa$ as
\[
a=\sum_{(s_1,s_2)\in S}v_{s_1s_2}(x_3)x_1^{s_1}x_2^{s_2}+h
\]
for $v_{s_1s_2}\in K[[x_3]]$ and $h\in K[[x_1,x_2,x_3]]$ with $\ord_\infty h>\ord_\infty\sfa$, where $\sO_{\hat X}=K[[x_1,x_2,x_3]]$. Take the minimum $c$ of the orders of $v_{s_1s_2}$ for $(s_1,s_2)\in S$. Then as far as $(\ord_\infty h-\ord_\infty\sfa)i$ is greater than $c$, one has $\ord_ia=\ord_i(\sum_{(s_1,s_2)\in S}v_{s_1s_2}x_1^{s_1}x_2^{s_2})=i\ord_\infty\sfa+c$ and hence $\ord_i\sfa^q=i\ord_\infty\sfa^q+qc$. Combining it with the inequality $\ord_i\sfa^q\ge i(\mu+1)$ above, one obtains
\[
i\ord_\infty\sfa^q+qc\ge i(\mu+1)
\]
for all large $i$. Thus $\ord_\infty\sfa^q\ge\mu+1$ and $a$ is contained in the ideal $\sfI$ in $\sO_{\hat X}$ generated by all monomials $x_1^{s_1}x_2^{s_2}$ such that $q(\mu s_1+s_2)\ge\mu+1$.

If $\mu$ is irrational, then the equality $q(\mu s_1+s_2)=\mu+1$ holds only if $(s_1,s_2)=(q^{-1},q^{-1})$. In this case, if $w_1/w_2$ is sufficiently close to $\mu$, then the ideal $\sfI$ is the same as that generated by all monomials $x_1^{s_1}x_2^{s_2}$ such that $q((w_1/w_2)s_1+s_2)\ge w_1/w_2+1$. No matter whether $\mu$ is rational, the divisor $\hat F$ obtained by the weighted blow-up of $\hat X$ with $\wt(x_1,x_2)=(w_1,w_2)$ has $\ord_{\hat F}\sfa^q\ge w_1+w_2$, that is, $a_{\hat F}(\hat X,\sfa^q)\le0$. Since $(\hat X,\sfa^q)$ is lc, $a_{\hat F}(\hat X,\sfa^q)$ is zero.
\end{proof}

\section{Admissible sequences}
In this section, we shall furnish an inductive framework by the lc slope for the proof of Theorem~\ref{thm:limit}. As before, let $x\in X$ be the germ of a smooth threefold and let $\fm$ be the maximal ideal in $\sO_X$ defining $x$. Fix a positive rational number $q$ and a non-negative rational number $s$. We also fix a positive integer $n$ such that $nq$ is an integer.

\begin{definition}
Let $\fa$ be an $\fm$-primary ideal in $\sO_X$ such that $(x\in X,\fa^q)$ is canonical of semistable type and such that $\mld_x(X,\fa^q\fm^s)$ is positive. An \textit{admissible} weighted blow-up $b\in F\subset B\to x\in X$ for $(X,\fa^q\fm^s)$ is the weighted blow-up of $X$ with $\wt(x_1,x_2,x_3)=(w_1,w_2,1)$ for some regular system $x_1,x_2,x_3$ of parameters in $\sO_{X,x}$ and for some pair $(w_1,w_2)$ of positive integers with $1\le w_1/w_2\le n$, endowed with the exceptional divisor $F$ and the point $b$ above $x$ in the strict transform of the curve in $X$ defined by $(x_1,x_2)\sO_X$, such that for the weak transform $\fa_B$ in $b\in B$ of $\fa$,
\begin{itemize}
\item
$(b\in B,\fa_B^q)$ is crepant to $(x\in X,\fa^q)$ and
\item
$\mld_x(X,\fa^q\fm^s)=\mld_b(B,\fa_B^q(\fm\sO_B)^s)$.
\end{itemize}
An admissible weighted blow-up $b\in F\subset B\to x\in X$ is said to be \textit{maximal} if there exists no admissible weighted blow-up $b'\in F'\subset B'\to x\in X$ such that $c_B(F')=b$.
\end{definition}

\begin{remark}\label{rmk:admissible}
Let $b\in F\subset B\to x\in X$ and $b'\in F'\subset B'\to x\in X$ be admissible weighted blow-ups for $(X,\fa^q\fm^s)$. If $c_B(F')=b$, then the strict inequality $a_F(X)<a_{F'}(X)$ of log discrepancies holds. Because $a_F(X)$ is bounded from above by the maximum of $a_G(X)$ for the finitely many divisors $G$ over $X$ that compute $\mld_x(X,\fa^q\fm^s)>0$, the existence of an admissible weighted blow-up for $(X,\fa^q\fm^s)$ guarantees that of a maximal one.
\end{remark}

Let $\cS=\{\fa_i\}_{i\in\bN_+}$ be a sequence of $\fm$-primary ideals in $\sO_X$ such that $(x\in X,\fa_i^q)$ is canonical of semistable type and such that $\mld_x(X,\fa_i^q\fm^s)$ is positive.

\begin{definition}
We say that $\{\fa_i^q\fm^s\}_{i\in\bN_+}$ is an \textit{admissible} sequence of \textit{slope} $\mu$ if for all $i$ there exists a maximal admissible weighted blow-up for $(X,\fa_i^q\fm^s)$, with respect to weights $(w_{1i},w_{2i},1)$, such that $\lim_{i\to\infty}w_{2i}=\infty$ and such that $\mu_i=w_{1i}/w_{2i}$ satisfies $1\le\mu_i\le n$ and has a limit $\mu=\lim_{i\to\infty}\mu_i$ in $\bR$. Note that $1\le\mu\le n$.
\end{definition}

\begin{lemma}\label{lem:admissible}
Given a sequence $\cS=\{\fa_i\}_{i\in\bN_+}$ as above, after passing to a subsequence, either one attains a positive integer $l$ such that for all $i$ there exists a divisor $E_i$ over $X$ which computes $\mld_x(X,\fa_i^q\fm^s)$ and has $a_{E_i}(X)\le l$, or one attains an admissible sequence $\{\fa_i^q\fm^s\}_{i\in\bN_+}$ of some slope.
\end{lemma}

\begin{proof}
Suppose that for infinitely many $i$, there exists a maximal admissible weighted blow-up for $(X,\fa_i^q\fm^s)$, with respect to weights $(w_{1i},w_{2i},1)$ for $1\le w_{1i}/w_{2i}\le n$. If the set of $w_{2i}$ is infinite, then passing to a subsequence, one attains an admissible sequence $\{\fa_i^q\fm^s\}_i$. From this observation together with Remark~\ref{rmk:admissible}, we may assume the existence of a constant $w$ independent of $i$ such that every admissible weighted blow-up for $(X,\fa_i^q\fm^s)$ is with respect to weights $(w_{1i},w_{2i},1)$ satisfying $w_{2i}<w$.

We construct a generic limit $\sfa$ of $\cS$ on $\hat x\in\hat X$. By Theorems~\ref{thm:special} and~\ref{thm:lcslope}, we may assume that $(\hat X,\sfa^q)$ is special and admits an lc slope $w_1/w_2$ for $(w_1,w_2)\in\sfP$. Then by Proposition~\ref{prp:down}, after replacement of $\cS$ by a subsequence, for all $i$ there exists a regular system $x_{1i},x_{2i},x_{3i}$ of parameters in $\sO_{X,x}$ such that $a_{F_i}(X,\fa_i^q)=1$ for the divisor $F_i$ obtained by the weighted blow-up of $X$ with $\wt(x_{1i},x_{2i},x_{3i})=(iw_1,iw_2,1)$. In order to derive the existence of the bound $l$, we may replace $\fa_i$ by an ideal in $\sO_{\bA^3}$, and hence $x\in X$ by $o\in\bA^3$, by the \'etale morphism $X\to\bA^3$ which pulls back the coordinates of $\bA^3$ to $x_{1i},x_{2i},x_{3i}$ as in Remark~\ref{rmk:etale}. This enables us to assume that $(x_{1i},x_{2i},x_{3i})$ is a common regular system $(x_1,x_2,x_3)$. Proposition~\ref{prp:up} confirms that $(\hat X,\sfa^q)$ remains special and that the $x_3$-axis is the smallest lc centre.

We shall apply Theorem~\ref{thm:success} for the weighted blow-up $b\in F\subset B\to x\in X$ with $\wt(x_1,x_2,x_3)=(ww_1,ww_2,1)$, where $b$ is the point above $x$ in the strict transform of the $x_3$-axis. As far as $i\ge w$, one has $\ord_F\fa_i^q\ge wi^{-1}\ord_{F_i}\fa_i^q=w(w_1+w_2)$ and hence $a_F(X,\fa_i^q)=1$ as $(X,\fa_i^q)$ is canonical. By the choice of $w$, the weighted blow-up $b\in F\subset B\to x\in X$ is never admissible for $(X,\fa_i^q\fm^s)$ and hence $\mld_x(X,\fa_i^q\fm^s)<\mld_b(B,\fa_{iB}^q(\fm\sO_B)^s)$. It thus follows from Theorem~\ref{thm:success} that $\mld_x(X,\fa_i^q\fm^s)=\mld_{\hat x}(\hat X,\sfa^q\hat\fm^s)$ for all $i$ after replacement of the family of approximations, where $\hat\fm$ is the maximal ideal in $\sO_{\hat X}$. Hence Lemma~\ref{lem:limit} yields a required bound $l$.
\end{proof}

From the above lemma, in order to prove Theorem~\ref{thm:limit}, we may assume that $\{\fa_i^q\fm^s\}_i$ is an admissible sequence of some slope $1\le\mu\le n$. First we treat the case when $\mu$ is not in the finite set $\sfQ_n$.

\begin{theorem}\label{thm:slopeR}
If $\{\fa_i^q\fm^s\}_i$ is an admissible sequence of slope not in $\sfQ_n$, then there exists a positive integer $l$ such that for infinitely many $i$, there exists a divisor $E_i$ over $X$ which computes $\mld_x(X,\fa_i^q\fm^s)$ and has $a_{E_i}(X)\le l$.
\end{theorem}

\begin{proof}
By the reduction to the case when $x\in X$ is the germ at origin of $\bA^3$ as in the proof of Lemma~\ref{lem:admissible}, we may assume the existence of a regular system $x_1,x_2,x_3$ of parameters in $\sO_{X,x}$ for which, for all $i$, there exists a maximal admissible weighted blow-up for $(X,\fa_i^q\fm^s)$ with $\wt(x_1,x_2,x_3)=(w_{1i},w_{2i},1)$ such that $\lim_{i\to\infty}w_{2i}=\infty$ and such that $\mu_i=w_{1i}/w_{2i}$ satisfies $1\le\mu_i\le n$ and has a limit $\mu=\lim_{i\to\infty}\mu_i$ in $\bR\setminus\sfQ_n$. Then Proposition~\ref{prp:up} yields, after replacement of $\cS$, an element $(w_1,w_2)$ in $\sfP\setminus\sfP_n$ such that $(\hat X,\sfa^q)$ admits an lc slope $w_1/w_2\not\in\sfQ_n$ with respect to $x_1,x_2,x_3$. On the other hand, it follows from Lemma~\ref{lem:lcslope} that $(\hat X,\sfa^q)$ also admits an lc slope in $\sfQ_n$ with respect to the same system $x_1,x_2,x_3$. Thus the theorem follows from Proposition~\ref{prp:lcslopes} and Lemma~\ref{lem:lct1}.
\end{proof}

In the next section, we shall treat the other case by induction on $\mu$.

\begin{theorem}\label{thm:slopeQ}
If $\{\fa_i^q\fm^s\}_i$ is an admissible sequence of slope $\mu\in\sfQ_n$, then there exists a positive integer $l$ such that for infinitely many $i$, there exists a divisor $E_i$ over $X$ which computes $\mld_x(X,\fa_i^q\fm^s)$ and has $a_{E_i}(X)\le l$.
\end{theorem}

We express $\mu=w_1/w_2$ by an element $(w_1,w_2)$ in $\sfP_n$. By induction, we assume the theorem for all admissible sequences of slope $\mu'\in\sfQ_n$ greater than $\mu$. As in the proofs of Lemma~\ref{lem:admissible} and Theorem~\ref{thm:slopeR}, we may assume the existence of a regular system $x_1,x_2,x_3$ of parameters in $\sO_{X,x}$ for which, for all $i$, there exists a maximal admissible weighted blow-up
\[
\pi_i\colon b_i\in F_i\subset B_i\to x\in X\quad\textrm{with}\quad\wt(x_1,x_2,x_3)=(w_{1i},w_{2i},1)
\]
for $(X,\fa_i^q\fm^s)$ such that $\lim_{i\to\infty}w_{2i}=\infty$ and such that $\mu_i=w_{1i}/w_{2i}$ satisfies $1\le\mu_i\le n$ and converges to $\mu=w_1/w_2$. Recall that $\mld_x(X,\fa_i^q\fm^s)=\mld_{b_i}(B_i,\fa_{iB}^q(\fm\sO_{B_i})^s)$ for the weak transform $\fa_{iB}$ in $b_i\in B_i$ of $\fa_i$.

\begin{lemma}\label{lem:slopeQ}
Under the inductive hypothesis by the slope, there exists a positive integer $l'$ such that for infinitely many $i$, there exists a divisor $E_i$ over $B_i$ which computes $\mld_{b_i}(B_i,\fa_{iB}^q(\fm\sO_{B_i})^s)$ and has $a_{E_i}(B_i)\le l'$.
\end{lemma}

\begin{proof}
\textit{Step}~1.
Take a divisor $\Gamma_i$ over $B_i$ which computes $\mld_{b_i}(B_i,\fa_{iB}^q(\fm\sO_{B_i})^s)$. It also computes $\mld_x(X,\fa_i^q\fm^s)$ and has $\ord_{\Gamma_i}\fm\le b$ for a positive integer $b$ depending only on $q$ and $s$, which is provided by Theorem~\ref{thm:lct} with the aid of Remark~\ref{rmk:etale}. Let $\fn_i$ denote the maximal ideal in $\sO_{B_i}$ defining $b_i$. For the $\fn_i$-primary ideals
\[
\fb_i=\fa_{iB}+\fn_i^{\ord_{\Gamma_i}\fa_{iB}},\qquad\fc_i=\fm\sO_{B_i}+\fn_i^b,
\]
one has the equality $a_{\Gamma_i}(B,\fa_{iB}^q(\fm\sO_B)^s)=a_{\Gamma_i}(B,\fb_i^q\fc_i^s)$. It follows that the inequality $\mld_{b_i}(B_i,\fa_{iB}^q(\fm\sO_{B_i})^s)\le\mld_{b_i}(B_i,\fb_i^q\fc_i^s)$ is actually an equality, that is,
\[
\mld_x(X,\fa_i^q\fm^s)=\mld_{b_i}(B_i,\fa_{iB}^q(\fm\sO_{B_i})^s)=\mld_{b_i}(B_i,\fb_i^q\fc_i^s).
\]

Fixing an \'etale morphism $\alpha_i\colon b_i\in B_i\to o\in A=\Spec k[a_1,a_2,a_3]$, we regard the ideals $\fb_i$ and $\fc_i$ as $\fn$-primary ideals in $\sO_A$, where $\fn$ is the maximal ideal in $\sO_A$ defining $o$, and construct a generic limit $(\sfb,\sfc)$ of $\{(\fb_i,\fc_i)\}_{i\in\bN_+}$ on $\hat o\in\hat A$. Then by Theorem~\ref{thm:special}, we may assume $(\hat A,\sfb^q)$ to be special. Further by Theorem~\ref{thm:meta} and Lemma~\ref{lem:admissible}, we may assume that for some $0\le b'\le b$, $\{\fb_i^q\fn^{sb'}\}_i$ is an admissible sequence on $o\in A$ of some slope $\mu'$. By Theorem~\ref{thm:slopeR} and the inductive hypothesis, we can assume that $\mu'\in\sfQ_n$ and $\mu'\le\mu$. Express $\mu'=v_1/v_2$ by an element $(v_1,v_2)\in\sfP_n$.

\textit{Step}~2.
We apply Proposition~\ref{prp:up} after a suitable choice of $\alpha_i$. It asserts that $(\hat A,\sfb^q)$ admits an lc slope $\mu'$ after replacement of $\cS$, in which we keep $(\hat A,\sfb^q)$ special. Then by Proposition~\ref{prp:down}, after replacement of $\cS$, for all $i$ there exists a regular system $y_{1i},y_{2i},y_{3i}$ of parameters in $\sO_{B_i,b_i}$ such that $a_{G_i}(B_i,\fb_i^q)=1$ for the divisor $G_i$ obtained by the weighted blow-up of $B_i$ with $\wt(y_{1i},y_{2i},y_{3i})=(iv_1,iv_2,1)$. We reset $\alpha_i$ as $\alpha_i^*(a_1,a_2,a_3)=(y_{1i},y_{2i},y_{3i})$, in which the $a_3$-axis is the smallest lc centre of $(\hat A,\sfb^q)$.

Let $b_i'\in G_i'\subset B_i'\to b_i\in B_i$ be the weighted blow-up with $\wt(y_{1i},y_{2i},y_{3i})=(v_1,v_2,1)$, where $b_i'\in B_i'$ is the point above $b_i$ in the strict transform of the curve in $B_i$ defined by $(y_{1i},y_{2i})\sO_{B_i}$. Then $\ord_{G_i'}\fb_i^q\ge i^{-1}\ord_{G_i}\fb_i^q=v_1+v_2$ for all $i$. Hence $a_{G_i'}(B_i,\fb_i^q)=1$ as $(B,\fb_i^q)$ is canonical. Further it follows from $\fa_{iB}\subset\fb_i$ that $a_{G_i'}(B_i,\fa_{iB}^q)=a_{G_i'}(B_i,\fb_i^q)=1$, as $(B_i,\fa_{iB}^q)$ is canonical. Moreover, this shows the equality $\ord_{G_i'}\fa_{iB}=\ord_{G_i'}\fb_i$ and hence the containment $\fa_{iB'}\subset\fb_{iB'}$ for the weak transforms $\fa_{iB'}$ and $\fb_{iB'}$ in $b_i'\in B_i'$ of $\fa_{iB}$ and $\fb_i$ respectively.

We apply Theorem~\ref{thm:success} to the sequence $\{(\fb_i,\fc_i)\}_i$ by identifying all the weighted blow-ups $b_i'\in G_i'\subset B_i'\to b_i\in B_i$ via $\alpha_i$. If $\mld_o(A,\fb_i^q\fc_i^s)=\mld_{\hat o}(\hat A,\sfb^q\sfc^s)$ for all $i$ after replacement of the family of approximations, then the required existence of $l'$ follows from Lemma~\ref{lem:limit}. Hence by Theorem~\ref{thm:success}, passing to a subsequence, we may and shall assume that for all $i$,
\[
\mld_x(X,\fa_i^q\fm^s)=\mld_{b_i}(B_i,\fb_i^q\fc_i^s)=\mld_{b_i'}(B_i',\fb_{iB'}^q(\fc_i\sO_{B_i'})^s).
\]
It follows that the inequality
\[
\mld_x(X,\fa_i^q\fm^s)\le\mld_{b_i'}(B_i',\fa_{iB'}^q(\fm\sO_{B_i'})^s)\le\mld_{b_i'}(B_i',\fb_{iB'}^q(\fc_i\sO_{B_i'})^s)
\]
is actually an equality. In particular, $\mld_x(X,\fa_i^q\fm^s)=\mld_{b_i'}(B_i',\fa_{iB'}^q(\fm\sO_{B_i'})^s)$.

\textit{Step}~3.
We shall derive a contradiction. We have the maximal admissible weighted blow-up $b_i\in F_i\subset B_i\to x\in X$ for $(X,\fa_i^q\fm^s)$ with $\wt(x_1,x_2,x_3)=(w_{1i},w_{2i},1)$ and the weighted blow-up $b_i'\in G_i'\subset B_i'\to b_i\in B_i$ with $\wt(y_{1i},y_{2i},y_{3i})=(v_1,v_2,1)$ such that $a_{G_i'}(B_i,\fa_{iB}^q)=1$. Recall that $\mu_i=w_{1i}/w_{2i}$ converges to $\mu=w_1/w_2$. Hence
\[
\Bigl\lfloor\frac{v_1-1}{v_2}\Bigr\rfloor\le\Bigl\lfloor\frac{w_1-1}{w_2}\Bigr\rfloor\le\frac{w_{1i}-n^2}{w_{2i}}\le\frac{w_{1i}-v_1^2}{w_{2i}}
\]
for large $i$, where the first inequality follows from $v_1/v_2=\mu'\le\mu=w_1/w_2$. Note that $\rd{(w_1-1)/w_2}=\rd{w_1/w_2}$ unless $w_2=1$. Thus one can apply Theorem~\ref{thm:composite}, which asserts that $G_i'$ is obtained by a certain weighted blow-up $Z_i\to X$ with weights $(w_{1i}+v_1-d_i,w_{2i}+v_2+d_i,1)$ for some integer $0\le d_i\le v_1-v_2$.

We describe $Z_i$ as the weighted blow-up of $X$ with $\wt(z_{1i},z_{2i},x_3)=(u_{1i},u_{2i},1)$ for the maximum $u_{1i}$ of $w_{1i}+v_1-d_i$ and $w_{2i}+v_2+d_i$ and for the minimum $u_{2i}$ of them. Since both $w_{1i}/w_{2i}$ and $(v_1-d_i)/(v_2+d_i)$ belong to the closed interval between $1/n$ and $n$, one has $1\le u_{1i}/u_{2i}\le n$. By Remark~\ref{rmk:composite}, the rational map $B_i'\dasharrow Z_i$ over $X$ is isomorphic about $b_i'$. In particular, the image $z_i\in Z_i$ of the point $b_i'$ is not contained in the strict transform of the divisor on $X$ defined by $x_3$. Hence by replacing $z_{ji}$ for $j=1,2$ by $z_{ji}+c_{ji}x_3^{u_{ji}}$ for some $c_{ji}\in k$, we may choose $z_{ji}$ so that $z_i$ lies on the strict transform of the curve in $X$ defined by $(z_{1i},z_{2i})\sO_X$.

The pair $(z_i\in Z_i,\fa_{iZ}^q)$ is crepant to $(x\in X,\fa^q)$ for the weak transform $\fa_{iZ}$ of $\fa_i$, and $\mld_x(X,\fa_i^q\fm^s)=\mld_{b_i'}(B_i',\fa_{iB'}^q(\fm\sO_{B_i'})^s)=\mld_{z_i}(Z_i,\fa_{iZ}^q(\fm\sO_{Z_i})^s)$. In other words, $Z_i\to X$ is an admissible weighted blow-up for $(X,\fa_i^q\fm^s)$. This contradicts the maximal choice of $B_i\to X$.
\end{proof}

\section{Weighted homogeneous degeneration}\label{sct:monomial}
This section is devoted to the proof of Theorem~\ref{thm:slopeQ}. By induction on the slope, we assume the theorem for sequences of slope greater than $\mu$. By Remark~\ref{rmk:etale}, we assume that $x\in X=\Spec k[x_1,x_2,x_3]$ and that there exists a maximal admissible weighted blow-up $\pi_i\colon b_i\in F_i\subset B_i\to x\in X$ with $\wt(x_1,x_2,x_3)=(w_{1i},w_{2i},1)$ for $1\le\mu_i=w_{1i}/w_{2i}\le n$ such that $\lim_{i\to\infty}w_{2i}=\infty$ and such that $\mu_i$ converges to $\mu\in\sfQ_n$, in which $\mld_x(X,\fa_i^q\fm^s)=\mld_{b_i}(B_i,\fa_{iB}^q(\fm\sO_{B_i})^s)$ for the weak transform $\fa_{iB}$ of $\fa_i$. By Lemma~\ref{lem:slopeQ}, we shall assume the existence of a divisor $E_i$ over $B_i$ which computes $\mld_x(X,\fa_i^q\fm^s)$ and has $c_{B_i}(E_i)=b_i$ and $a_{E_i}(B_i)\le l'$ for the constant $l'$ provided in the lemma. Replacing $\fa_i$ by $\fa_i^{nq}$, we shall assume that $q=1/n$. We may assume that $n\ge2$. We express $\mu=w_1/w_2$ by $(w_1,w_2)\in\sfP_n$.

By the existence of the limit $\mu$, the finite set
\[
S=\{(s_1,s_2)\in\bN^2\mid s_1+s_2<(\mu_i+1)n+\mu_i\}
\]
is independent of $i$ after replacement of $\cS$ by a subsequence. Since there are only finitely many ways of decomposition of $S$ into two subsets, the decomposition into $S^-=\{(s_1,s_2)\in S\mid\mu_is_1+s_2\le(\mu_i+1)n\}$ and $S^+=S\setminus S^-$ is also independent of $i$ after replacement of $\cS$. For each $(s_1,s_2)\in S^-$, we define the non-negative integer $s_{3i}$ depending on $i$ as
\[
s_{3i}=(w_{1i}+w_{2i})n-(w_{1i}s_1+w_{2i}s_2).
\]
To avoid intricacy, we do not add further subscript to the symbol $s_{3i}$ in spite of the dependence on $(s_1,s_2)$. Passing to a subsequence, we assume that for each $(s_1,s_2)\in S^-$, either $s_{3i}$ is constant, say $s_3$ (which depends on $(s_1,s_2)$), or $s_{3i}$ diverges to infinity. Accordingly we decompose $S^-$ as the disjoint union of
\[
S_0^-=\{(s_1,s_2)\in S^-\mid s_{3i}=s_3\},\qquad S_\infty^-=\{(s_1,s_2)\in S^-\mid\textstyle\lim_is_{3i}=\infty\}.
\]
Note that $(n,n)$ belongs to $S_0^-$. We assume the decomposition $S=S_0^-\sqcup S_\infty^-\sqcup S^+$ henceforth.

\begin{lemma}\label{lem:S}
\begin{enumerate}
\item\label{itm:S-plus}
There exists a rational number $\mu'\ge1$ other than $\mu$ such that $rs_1+s_2\ge(r+1)n$ for all $(s_1,s_2)\in S^+$ and all real numbers $r$ that belong to the closed interval between $\mu'$ and $\mu$.
\item\label{itm:S-ratl}
If $S_0^-\neq\{(n,n)\}$, then $(w_{1i},w_{2i})=(w_{11},w_{21})+c_i(w_1,w_2)$ for some integer $c_i$ for all $i$.
\end{enumerate}
\end{lemma}

\begin{proof}
An element $(s_1,s_2)$ in $S^+$ satisfies $\mu_is_1+s_2>(\mu_i+1)n$ for all $i$. Taking the limit, one obtains $\mu s_1+s_2\ge(\mu+1)n$. If $\mu_1=\mu$, then the strict inequality $\mu s_1+s_2>(\mu+1)n$ yields the existence of $\mu'$. If $\mu_1\neq\mu$, then the inequality $rs_1+s_2\ge(r+1)n$ holds for all $r$ between $\mu_1$ and $\mu$ because it does for both $r=\mu_1$ and $\mu$. This is the first assertion.

In order to see the second assertion, we shall assume the existence of an element $(s_1,s_2)$ in $S_0^-$ other than $(n,n)$. It satisfies the equality $w_{1i}(s_1-n)+w_{2i}(s_2-n)+s_3=0$. Using the greatest common divisor $g$ of $\abs{s_1-n}$ and $\abs{s_2-n}$, which also divides $s_3$, we write $s_j-n=gs_j'$ for $j=1,2$ and $s_3=gs'_3$. Then $w_{1i}s_1'+w_{2i}s_2'+s_3'=0$. Since $w_{1i}$ and $w_{2i}$ diverge to infinity, $s_1's_2'$ is negative and one can write $(w_{1i},w_{2i})= (w_{11},w_{21})+c_i(\abs{s_2'},\abs{s_1'})$ by some integer $c_i$. Then $\mu=\lim_iw_{1i}/w_{2i}=\abs{s_2'}/\abs{s_1'}$. Thus $(\abs{s_2'},\abs{s_1'})=(w_1,w_2)$.
\end{proof}

We define the ideal
\[
\fA_i=\pi_{i*}\sO_{B_i}(-(w_{1i}+w_{2i})nF_i)
\]
in $\sO_X$ generated by all monomials $x_1^{s_1}x_2^{s_2}x_3^{s_3}$ such that $w_{1i}s_1+w_{2i}s_2+s_3\ge(w_{1i}+w_{2i})n$. It follows from $a_{F_i}(X,\fa_i^q)=1$ that $\fa_i\subset\fA_i$. We construct a generic limit $(\sfa,\sfA)$ of the sequence $\{(\fa_i,\fA_i)\}_{i\in\bN_+}$ using the notation in Section~\ref{sct:limit}. The limit $(\sfa,\sfA)$ is with respect to a family $\cF=(Z_l,(\fa(l),\fA(l)),N_l,s_l,t_l)_{l\ge l_0}$ of approximations of $\{(\fa_i,\fA_i)\}_i$. The ideals $\sfa$ and $\sfA$ are defined on $\hat x\in\hat X=\Spec K[[x_1,x_2,x_3]]$ and satisfy $\sfa\subset\sfA$. We write $\hat\fm$ for the maximal ideal in $\sO_{\hat X}$. By Proposition~\ref{prp:up}, we shall assume that $(\hat X,\sfa^q)$ is special and admits an lc slope $\mu=w_1/w_2$ with respect to $x_1,x_2,x_3$.

We provide generators of $\fA_i$ and $\sfA$ in terms of the decomposition of $S$.

\begin{lemma}\label{lem:generators}
The ideal $\fA_i$ is generated by all $x_1^{s_1}x_2^{s_2}x_3^{s_{3i}}$ for $(s_1,s_2)\in S^-$ and all $x_1^{s_1}x_2^{s_2}$ for $(s_1,s_2)\in S^+$, whilst the limit $\sfA$ is generated by all $x_1^{s_1}x_2^{s_2}x_3^{s_3}$ for $(s_1,s_2)\in S_0^-$ and all $x_1^{s_1}x_2^{s_2}$ for $(s_1,s_2)\in S^+$.
\end{lemma}

\begin{proof}
We write $x^s$ for the monomial $x_1^{s_1}x_2^{s_2}x_3^{s_3}$ for brevity. If $x^s$ satisfies $w_{1i}s_1+w_{2i}s_2+s_3\ge(w_{1i}+w_{2i})n+w_{1i}$, then the monomial $x^sx_j^{-1}$ with $s_j\neq0$ already belongs to $\fA_i$. Hence $\fA_i$ is generated by all $x^s$ such that $(w_{1i}+w_{2i})n\le w_{1i}s_1+w_{2i}s_2+s_3<(w_{1i}+w_{2i})n+w_{1i}$. Any such monomial satisfies
\[
s_1+s_2\le\mu_is_1+s_2+w_{2i}^{-1}s_3<(\mu_i+1)n+\mu_i
\]
and thus $(s_1,s_2)$ belongs to the set $S$.

For each $(s_1,s_2)\in S$, the least non-negative integer $s_3$ such that $x^s\in\fA_i$ is $s_{3i}$ if $(s_1,s_2)\in S^-$ and is zero if $(s_1,s_2)\in S^+$. This is the assertion for $\fA_i$. The assertion for the generic limit $\sfA$ is derived at once from that for $\fA_i$. Indeed, let $\fA$ be the ideal in $\sO_X$ generated by all $x_1^{s_1}x_2^{s_2}x_3^{s_3}$ for $(s_1,s_2)\in S_0^-$ and all $x_1^{s_1}x_2^{s_2}$ for $(s_1,s_2)\in S^+$. Then $\fA\subset\fA_i$ for all $i$, and for each $l$, one has $\fA_i\subset\fA+\fm^l$ for all $i$ that satisfy $l\le s_{3i}$ for all $(s_1,s_2)\in S_\infty^-$. Thus $\fA\sO_{\hat X}\subset\sfA\subset\varprojlim_l((\fA+\fm^l)\otimes_kK)/(\fm^l\otimes_kK)=\fA\sO_{\hat X}$.
\end{proof}

\begin{lemma}\label{lem:unique}
If $S_0^-=\{(n,n)\}$, then $(\hat X,\sfa^q\hat\fm)$ is lc.
\end{lemma}

\begin{proof}
Take any $(v_1,v_2)\in\sfP$ such that $v_1/v_2$ belongs to the interval between $\mu'$ and $\mu$ in Lemma~\ref{lem:S}(\ref{itm:S-plus}). Then $v_1s_1+v_2s_2\ge(v_1+v_2)n$ for all $(s_1,s_2)\in S^+$. Let $\hat G$ be the divisor obtained by the weighted blow-up of $\hat X$ with $\wt(x_1,x_2)=(v_1,v_2)$. If $S_0^-=\{(n,n)\}$, then by Lemma~\ref{lem:generators}, $\sfA$ is generated by $x_1^nx_2^n$ and all $x_1^{s_1}x_2^{s_2}$ for $(s_1,s_2)\in S^+$. Hence the order $\ord_{\hat G}\sfA$, which is the minimum of $(v_1+v_2)n$ and all $v_1s_1+v_2s_2$ for $(s_1,s_2)\in S^+$, equals $(v_1+v_2)n$, that is, $a_{\hat G}(\hat X,\sfA^q)=0$. Since $(\hat X,\sfa^q)$ is lc, it follows from $\sfa\subset\sfA$ that $a_{\hat G}(\hat X,\sfa^q)=a_{\hat G}(\hat X,\sfA^q)=0$. In other words, $(\hat X,\sfa^q)$ admits an lc slope $v_1/v_2$ with respect to $x_1,x_2,x_3$. Now the lemma follows from Proposition~\ref{prp:lcslopes}.
\end{proof}

Henceforth we assume the existence of an element in $S_0^-$ other than $(n,n)$. Then by Lemma~\ref{lem:S}(\ref{itm:S-ratl}), we can write
\[
(w_{1i},w_{2i})=(w_{11},w_{21})+c_i(w_1,w_2)
\]
by some integer $c_i$, which diverges to infinity. Let $\varphi\colon Y\to X$ and $\hat\varphi\colon\hat Y\to\hat X$ be the weighted blow-ups with $\wt(x_1,x_2)=(w_1,w_2)$ and let $F$ and $\hat F$ be the exceptional divisors of $\varphi$ and $\hat\varphi$ respectively. Recall that $(\hat X,\sfa^q)$ admits an lc slope $w_1/w_2$ with respect to $x_1,x_2,x_3$, namely $\ord_{\hat F}\sfa=(w_1+w_2)n$. Define
\[
\sI=\varphi_*\sO_Y(-((w_1+w_2)n+1)F),\qquad\hat\sI=\hat\varphi_*\sO_{\hat Y}(-((w_1+w_2)n+1)\hat F).
\]

Let $f_i$ be the general member of $\fa_i$. One can assume that the generic limit $f_\infty$ of the sequence $\{f_i\}_i$ of functions in $\sO_X$ is defined in $\sO_{\hat X}$ \cite{Koaxv}. The function $f_\infty$ in $\sO_{\hat X}$ is the inverse limit $\varprojlim_l f(l)$ of the approximated functions $f(l)$ in $\sO_{X\times Z_l}$ which can be taken in the $\sO_{Z_l}$-module $\sum_s\sO_{Z_l}x^s$ generated by all monomials $x^s$ in $x_1,x_2,x_3$ of total degree less than $l$, and $f_i$ is congruent modulo $\fm^l$ to $f(l)_i$ given by $s_l(i)\in Z_l$. We use the same notation $\cF$ for the family of approximations of the sequence $\{(f_i,\fa_i,\fA_i)\}_i$.

By abuse of notation, we shall write $\theta^q$ for the $\bQ$-ideal $(\theta\sO_Z)^q$ for a function $\theta$ on a scheme $Z$, such as $f_i^q=(f_i\sO_X)^q$. Recall that $q=1/n<1$. Since a log resolution of $(X,\fa_i\fm)$ is also a log resolution of $(X,f_i\fm)$, one has $\mld_x(X,\fa_i^q\fm^t)=\mld_x(X,f_i^q\fm^t)$ for all $t\ge0$, and provided that $(X,\fa_i^q\fm^t)$ is lc, a divisor over $X$ computes $\mld_x(X,\fa_i^q\fm^t)$ if and only if it computes $\mld_x(X,f_i^q\fm^t)$. In particular, $\ord_{F_i}\fa_i=\ord_{F_i}f_i$ and $\ord_{E_i}\fa_i=\ord_{E_i}f_i$. From this observation, one can replace $\fa_i$ by $f_i\sO_X+\fm^{n_i}$ for an integer $n_i\ge i$ such that $\fm^{n_i}\subset\fa_i$, by which $\fa_i+\fm^l=f_i\sO_X+\fm^l$ for all $i\ge l$. Consequently one can take $\fa(l)$ to be $f(l)\sO_{X\times Z_l}+\fm^l\sO_{X\times Z_l}$ after shrinking $Z_l$ and then $\sfa=\varprojlim_l\fa(l)_K/(\fm^l\otimes_kK)=f_\infty\sO_{\hat X}$. Thus we may and shall assume that $\sfa$ is the principal ideal generated by $f_\infty$.

Define the finite set $T=\{(t_1,t_2)\in\bN^2\mid w_1t_1+w_2t_2\le(w_1+w_2)n\}$, for which $(t_1,t_2)\in T$ if and only if $x_1^{t_1}x_2^{t_2}\not\in\sI$. We adopt the simplified notation $t=(t_1,t_2)$ and $x^t=x_1^{t_1}x_2^{t_2}$. We truncate $f_i$ by $\sI$ and obtain the polynomial
\[
g_i=\sum_{t\in T}p_{it}x^t
\]
for $p_{it}\in k[x_3]$ such that $f_i$ is congruent to $g_i$ modulo $\sI$. We can assume that for each $t\in T$, the generic limit $p_{\infty t}$ of the sequence $\{p_{it}\}_i$ is defined in $K[[x_3]]$. Then
\[
g_\infty=\sum_{t\in T}p_{\infty t}(x_3)x^t
\]
is the generic limit of the sequence $\{g_i\}_i$ and $f_\infty$ is congruent to $g_\infty$ modulo $\hat\sI$.

\begin{lemma}\label{lem:fg}
If $S_0^-\neq\{(n,n)\}$, then
\begin{enumerate}
\item\label{itm:fg-ord}
$\ord_{\hat F}\sfa=\ord_{\hat F}g_\infty=(w_1+w_2)n$,
\item\label{itm:fg-infty}
$\mld_{\hat x}(\hat X,\sfa^q\hat\fm^s)=\mld_{\hat x}(\hat X,g_\infty^q\hat\fm^s)$, and
\item\label{itm:fg-seq}
$\mld_x(X,\fa_i^q\fm^s)\ge\mld_x(X,g_i^q\fm^s)$ for all $i$ but finitely many indices.
\end{enumerate}
\end{lemma}

\begin{proof}
From $\sfa=f_\infty\sO_{\hat X}$, one has $\ord_{\hat F}f_\infty=\ord_{\hat F}\sfa=(w_1+w_2)n$. Then $\ord_{\hat F}f_\infty=\ord_{\hat F}g_\infty$ since $f_\infty$ is congruent to $g_\infty$ modulo $\hat\sI$. This is the first assertion. Moreover, the weak $\bQ$-transforms in $\hat Y$ of $f_\infty$ and $g_\infty$ have the same restriction to $\hat F$. Thus the second assertion is derived from Lemma~\ref{lem:pia}(\ref{itm:pia-3fold}).

We have a divisor $E_i$ over $B_i$ which computes $\mld_x(X,\fa_i^q\fm^s)$ and has $c_{B_i}(E_i)=b_i$ and $a_{E_i}(B_i)\le l'$. In order to see the third assertion, we shall prove the inequality $\ord_{E_i}\sI>\ord_{E_i}\fa_i$ for all but finitely many $i$. Since $f_i$ is congruent to $g_i$ modulo $\sI$, this yields the equality $\ord_{E_i}f_i=\ord_{E_i}g_i$ and
\[
\mld_x(X,\fa_i^q\fm^s)=a_{E_i}(X,f_i^q\fm^s)=a_{E_i}(X,g_i^q\fm^s)\ge\mld_x(X,g_i^q\fm^s).
\]

From $(w_{1i},w_{2i})=(w_{11},w_{21})+c_i(w_1,w_2)$, one has the inequality $\ord_{F_i}\ge c_i\ord_F$ of order functions and hence
\[
\ord_{F_i}\sI\ge c_i\ord_F\sI\ge c_i((w_1+w_2)n+1)=(w_{1i}+w_{2i})n+c_i-e
\]
for the constant $e=(w_{11}+w_{21})n$. Since $\ord_{F_i}\fa_i=(w_{1i}+w_{2i})n$, one has
\[
\ord_{E_i}\sI\ge\ord_{F_i}\sI\cdot\ord_{E_i}{F_i}\ge(\ord_{F_i}\fa_i+c_i-e)\ord_{E_i}{F_i}.
\]
Hence as far as $c_i\ge e$, one obtains
\[
\ord_{E_i}\sI\ge\ord_{F_i}\fa_i\cdot\ord_{E_i}{F_i}+c_i-e=\ord_{E_i}\fa_i-\ord_{E_i}\fa_{iB}+c_i-e.
\]
Because $(b_i\in B_i,\fa_{iB}^q)$ is canonical, one also has $\ord_{E_i}\fa_{iB}\le(a_{E_i}(B_i)-1)n<l'n$. Thus $\ord_{E_i}\sI>\ord_{E_i}\fa_i-l'n+c_i-e$, and the inequality $\ord_{E_i}\sI>\ord_{E_i}\fa_i$ holds once $c_i$ exceeds the constant $l'n+e$.
\end{proof}

It follows from Lemma~\ref{lem:fg}(\ref{itm:fg-ord}) that $g_\infty=\sum_{t\in T}p_{\infty t}x^t$ is weighted homogeneous of weighted degree $(w_1+w_2)n$ with respect to $\wt(x_1,x_2)=(w_1,w_2)$. In other words, $g_\infty=\sum_{t\in T_0}p_{\infty t}x^t$ for the subset $T_0=\{(t_1,t_2)\in\bN^2\mid w_1t_1+w_2t_2=(w_1+w_2)n\}$ of $T$. Thus the generic limit $g_\infty$ of $\{g_i\}_i$ is also the generic limit of the sequence $\{h_i\}_i$ of
\[
h_i=\sum_{t\in T_0}p_{it}x^t.
\]

\begin{lemma}\label{lem:gh}
If $S_0^-\neq\{(n,n)\}$, then after replacing $\cF$, one has $\mld_{\hat x}(\hat X,g_\infty^q\hat\fm^s)=\mld_x(X,h_i^q\fm^s)\le\mld_x(X,g_i^q\fm^s)$ for all $i\in N_{l_0}$.
\end{lemma}

\begin{proof}
The action on $X$ of the multiplicative group $k^\times=\Spec k[\lambda,\lambda^{-1}]$ which sends $(x_1,x_2,x_3)$ to $(\lambda^{-w_1}x_1,\lambda^{-w_2}x_2,x_3)$ induces a flat degeneration $h_i$ of $g_i$ at $\lambda=0$. Hence the inequality $\mld_x(X,h_i^q\fm^s)\le\mld_x(X,g_i^q\fm^s)$ follows from the precise inversion of adjunction and lower semi-continuity of minimal log discrepancies on smooth varieties \cite{EMY03}. Precisely speaking, for the function $\tilde g_i=\lambda^{(w_1+w_2)n}g_i(\lambda^{-w_1}x_1,\lambda^{-w_2}x_2,x_3)$ on the family $\tilde X=X\times\bA^1\to\bA^1=\Spec k[\lambda]$, the lower semi-continuity
\[
\mld_{(x,0)}(\tilde X,\tilde g_i^q\tilde\fm^s)\le\mld_{(x,\lambda)}(\tilde X,\tilde g_i^q\tilde\fm^s)
\]
holds for general $\lambda\in\bA^1$, where $\tilde\fm$ stands for the ideal sheaf in $\sO_{\tilde X}$ defining $x\times\bA^1$. Let $X_\lambda=X\times\lambda$ denote the fibre of $\tilde X$ at $\lambda$. From the precise inversion of adjunction, one deduces
\[
\mld_x(X,h_i^q\fm^s)=\mld_{(x,0)}(\tilde X,X_0,\tilde g_i^q\tilde\fm^s)\le\mld_{(x,0)}(\tilde X,\tilde g_i^q\tilde\fm^s)-1.
\]
On the other hand, one sees on a log resolution of $(\tilde X,\tilde g_i\tilde\fm)$ that
\[
\mld_x(X,g_i^q\fm^s)=\mld_x(X_\lambda,(\tilde g_i\sO_{X_\lambda})^q(\tilde\fm\sO_{X_\lambda})^s)=\mld_{(x,\lambda)}(\tilde X,\tilde g_i^q\tilde\fm^s)-1
\]
for general $\lambda$. These are combined into the desired inequality.

We need to prove the equality $\mld_{\hat x}(\hat X,g_\infty^q\hat\fm^s)=\mld_x(X,h_i^q\fm^s)$. Recall that for each $t\in T_0$, the formal power series $p_{\infty t}\in K[[x_3]]$ is the generic limit of $\{p_{it}\}_i$, that is, $p_{\infty t}=\varprojlim_l p_t(l)$ for the approximated functions $p_t(l)$ taken in $\sum_{j<l}\sO_{Z_l}x_3^j$. Let
\[
h(l)=\sum_{t\in T_0}p_t(l)x^t.
\]
Similarly to Lemma~\ref{lem:resolution}, one has $\mld_{\hat x}(\hat X,(h(l)\sO_{\hat X})^q\hat\fm^s)=\mld_x(X,h(l)_i^q\fm^s)$ for all $i\in N_l$ and $l\ge l_0$ after replacement of $\cF$, where $h(l)_i$ denotes the fibre of $h(l)$ at $s_l(i)$. By the description $g_\infty=\varprojlim_l h(l)$, replacing $l_0$ by a greater integer, one can assume that $\mld_{\hat x}(\hat X,g_\infty^q\hat\fm^s)=\mld_{\hat x}(\hat X,(h(l)\sO_{\hat X})^q\hat\fm^s)$ for all $l\ge l_0$. Thus it suffices to show the existence of $l$ such that $\mld_x(X,h(l)_i^q\fm^s)=\mld_x(X,h_i^q\fm^s)$ for all $i\in N_l$.

Let $\fb(l)_i$ and $\fb_i$ be the weak transforms in $Y$ of $h(l)_i^{w_1w_2}\sO_X$ and $h_i^{w_1w_2}\sO_X$. Let $q'=q/w_1w_2$. Then $(Y,F,\fb(l)_i^{q'}(\fm\sO_Y)^s)$ is crepant to $(X,h(l)_i^q\fm^s)$ whilst $(Y,F,\fb_i^{q'}(\fm\sO_Y)^s)$ is crepant to $(X,h_i^q\fm^s)$. Let $\Delta=(1-w_2^{-1})D_1+(1-w_1^{-1})D_2$ be the different on $F$ defined by $(K_Y+F)|_F=K_F+\Delta$ as in the proof of Proposition~\ref{prp:lcslopes}, where $D_j$ is the support of the restriction to $F$ of the strict transform in $Y$ of the $x_{3-j}x_3$-plane. Let $f$ be the fibre of $\varphi\colon Y\to X$ at $x$. By Lemma~\ref{lem:pia}(\ref{itm:pia-3fold}), which also holds for the affine space $X$, one has
\begin{align*}
\mld_x(X,h(l)_i^q\fm^s)&=\mld_f(F,\Delta+sf,(\fb(l)_i\sO_F)^{q'}),\\
\mld_x(X,h_i^q\fm^s)&=\mld_f(F,\Delta+sf,(\fb_i\sO_F)^{q'}).
\end{align*}

On the other hand, the congruence of $h(l)_i$ and $h_i$ modulo $x_3^l$ implies that of $h(l)_i^{w_1w_2}$ and $h_i^{w_1w_2}$ modulo $x_3^l$, because $h(l)_i-h_i$ is a factor of $h(l)_i^{w_1w_2}-h_i^{w_1w_2}$. Hence $\fb(l)_i\sO_F+\sO_F(-lf)=\fb_i\sO_F+\sO_F(-lf)$. Thus by the uniform $\fm$-adic semi-continuity of minimal log discrepancies on $F$, which follows from Remark~\ref{rmk:equiv}(\ref{itm:equiv-sf}), there exists a positive integer $l_1$ depending only on $q$, $s$ (and $w_1,w_2$) such that
\[
\mld_f(F,\Delta+sf,(\fb(l)_i\sO_F)^{q'})=\mld_f(F,\Delta+sf,(\fb_i\sO_F)^{q'})
\]
for all $i\in N_l$ and $l\ge l_1$. Namely, $\mld_x(X,h(l)_i^q\fm^s)=\mld_x(X,h_i^q\fm^s)$.
\end{proof}

\begin{proof}[Proof of Theorem~\textup{\ref{thm:slopeQ}}]
If $(n,n)$ is a unique element in $S_0^-$, then the assertion follows from Lemmata~\ref{lem:lct1} and~\ref{lem:unique}. If $S_0^-$ has an element other than $(n,n)$, then by Lemmata~\ref{lem:fg}(\ref{itm:fg-infty})(\ref{itm:fg-seq}) and~\ref{lem:gh}, one attains $\mld_{\hat x}(\hat X,\sfa^q\hat\fm^s)\le\mld_x(X,\fa_i^q\fm^s)$ for all $i\in N_{l_0}$ after replacement of $\cF$, and the assertion follows from Lemma~\ref{lem:limit} and the discussion prior to that lemma. By induction on the slope $\mu\in\sfQ_n$, the theorem is completed.
\end{proof}

\begin{proof}[Proof of Theorem~\textup{\ref{thm:limit}}]
By Lemma~\ref{lem:admissible}, we may assume that $\{\fa_i^q\fm^s\}_i$ is an admissible sequence. Thus the theorem follows from Theorems~\ref{thm:slopeR} and~\ref{thm:slopeQ}.
\end{proof}

\section{Proof of theorems}
We shall complete the proof of the main theorems and corollary.

\begin{proof}[Proof of Theorems~\textup{\ref{thm:main}},~\textup{\ref{thm:alc}},~\textup{\ref{thm:adic}},~\textup{\ref{thm:nakamura}} and~\textup{\ref{thm:product}}]
Theorem~\ref{thm:limit} implies Theorem~\ref{thm:product}. By \cite[theorem~1.1]{K21} and Proposition~\ref{prp:HLL} with Remark~\ref{rmk:etale}, Theorem~\ref{thm:product} implies Theorems~\ref{thm:main} to~\ref{thm:nakamura}.
\end{proof}

\begin{proposition}[Han--Liu--Luo \cite{HLLaxv}]\label{prp:HLL}
The four statements in Theorem~\textup{\ref{thm:equiv}} imply the following for a subset $I$ of the positive real numbers which satisfies the DCC\@.
\begin{enumerate}
\item
There exists a positive integer $l$ depending only on $X$ and $I$ such that for $\bR$-ideals $\fa=\prod_{j=1}^e\fa_j^{r_j}$ and $\fb=\prod_{j=1}^e\fb_j^{r_j}$ on $X$, if $r_j\in I$ and $\fa_j+\fm^l=\fb_j+\fm^l$ for all $j$, where $\fm$ is the maximal ideal in $\sO_X$ defining $x$, then $\mld_x(X,\fa)=\mld_x(X,\fb)$.
\item
There exists a positive integer $l$ depending only on $X$ and $I$ such that for an $\bR$-ideal $\fa$ on $X$, if $\fa\in I$, then there exists a divisor $E$ over $X$ which computes $\mld_x(X,\fa)$ and has $a_E(X)\le l$.
\end{enumerate}
\end{proposition}

\begin{proof}
We write $r$ for the minimum of the elements in $I$.

\textit{Step}~1.
First of all, we shall deduce the first assertion from the second one. For an $\bR$-ideal $\fa=\prod_{j=1}^e\fa_j^{r_j}$ on $X$ with $r_j\in I$, the second assertion provides a divisor $E$ over $X$ which computes $\mld_x(X,\fa)$ and has $a_E(X)\le l$. Set $\fc=\prod_{j=1}^e(\fa_j+\fm^{l_1})^{r_j}$ for $l_1=\rd{r^{-1}l}+1$. It suffices to show that $\mld_x(X,\fc)=\mld_x(X,\fa)$.

If $(X,\fa)$ is lc, then $\ord_E\fa\le a_E(X)\le l$ and $\ord_E\fa_j\le r_j^{-1}l<l_1\le\ord_E\fm^{l_1}$. Hence $a_E(X,\fc)=a_E(X,\fa)$ and $E$ computes $\mld_x(X,\fc)=\mld_x(X,\fa)$. If $(X,\fa)$ is not lc, then either $a_E(X,\fc)=a_E(X,\fa)<0$ or $\ord_E\fm^{l_1}<\ord_E\fa_j$ for some $j$. In the latter case, $a_E(X,\fc)\le a_E(X)-r_j\ord_E\fm^{l_1}\le l-rl_1<0$. Eventually, $\mld_x(X,\fc)=-\infty=\mld_x(X,\fa)$.

\textit{Step}~2.
We proceed to the proof of the second assertion. Consider an $\bR$-ideal $\fa=\prod_{j=1}^e\fa_j^{r_j}$ on $X$ such that $r_j\in I$ and $\fa_j$ is non-trivial for all $j$. Then $\mld_x(X,\fa)\le\mld_xX-\sum_{j=1}^er_j\le\mld_xX-er$. Take $e_0=\rd{r^{-1}\mld_xX}+1$. If $(X,\fa)$ is lc, then $e<e_0$ and $r_j<e_0r$. If $(X,\fa)$ is not lc, then one can replace $\fa$ by $\fa'=\prod_{j=1}^{e'}(\fa_j)^{r_j'}$ for $e'=\min\{e,e_0\}$ and $r_j'=\min\{r_j,e_0r\}$, and hence add $e_0r$ to $I$, because every divisor computing $\mld_x(X,\fa')=-\infty$ also computes $\mld_x(X,\fa)$. From this observation, one has only to deal with $\bR$-ideals $\fa=\prod_{j=1}^e\fa_j^{r_j}$ for a fixed integer $e$ and for bounded exponents $r_j\le e_0r$ in $I$.

Let $\{\fa_i\}_{i\in\bN_+}$ be an arbitrary sequence of $\bR$-ideals $\fa_i=\prod_{j=1}^e\fa_{ij}^{r_{ij}}$ on $X$ such that $r_{ij}\in I$ and $r_{ij}\le e_0r$ for all $i$ and $j$. It suffices to find a positive integer $l$ such that for infinitely many indices $i$, there exists a divisor $E_i$ over $X$ which computes $\mld_x(X,\fa_i)$ and has $a_{E_i}(X)\le l$. Passing to a subsequence, we may assume that $\{r_{ij}\}_i$ is a non-decreasing sequence and has a limit $r_j=\lim_{i\to\infty}r_{ij}$ in $\bR$ for all $j$. Set $\fb_i=\prod_{j=1}^e\fa_{ij}^{r_j}$. By the assumption for the finite set $\{r_1,\ldots,r_e\}$, there exists a positive integer $a$ such that for all $i$, there exists a divisor $F_i$ over $X$ which computes $\mld_x(X,\fb_i)$ and has $a_{F_i}(X)\le a$.

Because the log discrepancy $a_{F_i}(X,\fb_i)=a_{F_i}(X)-\ord_{F_i}\fb_i$ belongs to the discrete set $a-n^{-1}\bN-\sum_jr_j\bN$ for the index $n$ of $x\in X$, there exists a positive real number $\delta_1$ such that if $(X,\fb_i)$ is not lc, then $a_{F_i}(X,\fb_i)$ is less than $-\delta_1$. In this case, for $\varepsilon_1=\delta_1/(\delta_1+a)$, one has
\[
a_{F_i}(X,\fb_i^{1-\varepsilon_1})=(1-\varepsilon_1)a_{F_i}(X,\fb_i)+\varepsilon_1a_{F_i}(X)<-(1-\varepsilon_1)\delta_1+\varepsilon_1a=0.
\]
Thus once $(1-\varepsilon_1)r_j\le r_{ij}\le r_j$ for all $j$, one has $a_{F_i}(X,\fa_i)\le a_{F_i}(X,\fb_i^{1-\varepsilon_1})<0$ and $F_i$ computes $\mld_x(X,\fa_i)=-\infty$.

\textit{Step}~3.
Henceforth we assume the log canonicity of $(X,\fb_i)$ for all $i$, in which $(X,\fa_i)$ is lc. Since $\mld_x(X,\fb_i)$ is bounded from above by $\mld_xX$, passing to a subsequence, the discreteness of log discrepancies \cite{K14} yields a non-negative real number $b$ and a positive real number $\delta$ such that $\mld_x(X,\fb_i)=a_{F_i}(X,\fb_i)=b$ for all $i$ and such that if a divisor $\Gamma$ over $X$ has $a_\Gamma(X,\fb_i)>b$, then $a_\Gamma(X,\fb_i)>b+\delta$. We may assume that $\delta\le1$. For $\varepsilon=\delta/a$, one has
\[
\mld_x(X,\fb_i^{1-\varepsilon})\le a_{F_i}(X,\fb_i^{1-\varepsilon})=(1-\varepsilon)a_{F_i}(X,\fb_i)+\varepsilon a_{F_i}(X)\le b+\varepsilon a=b+\delta.
\]
Thus once $(1-\varepsilon)r_j\le r_{ij}\le r_j$ for all $j$, one has $\mld_x(X,\fa_i)\le\mld_x(X,\fb_i^{1-\varepsilon})\le b+\delta$. Consequently if a divisor $E_i$ over $X$ computes $\mld_x(X,\fa_i)$, then $a_{E_i}(X,\fb_i)\le a_{E_i}(X,\fa_i)\le b+\delta$ and thus $a_{E_i}(X,\fb_i)=b$ by the choice of $\delta$. In other words, every divisor $E_i$ computing $\mld_x(X,\fa_i)$ also computes $\mld_x(X,\fb_i)$.

The rest of the proof is exactly the same as the proof of \cite[lemma~6.11]{HLLaxv}. We shall explain it briefly. Reordering indices and passing to a subsequence, we may assume the existence of $0\le e_1\le e$ such that $\{r_{ij}\}_i$ is strictly increasing for all $1\le j\le e_1$ and such that $r_{ij}=r_j$ for all $i$ and $e_1<j\le e$. We shall prove the existence of $l$ by induction on $e_1$. The claim for $e_1=0$ is assumed.

Assume that $e_1\ge1$ and set $\fb_{ij}=\fa_{ij}^{r_j}\prod_{k\neq j}\fa_{ik}^{r_{ik}}$ for $1\le j\le e_1$. By the inductive hypothesis, there exists a positive integer $l'$ such that for all $i$ and $1\le j\le e_1$, there exists a divisor $F_{ij}$ over $X$ which computes $\mld_x(X,\fb_{ij})$ and has $a_{F_{ij}}(X)\le l'$. Using the log canonicity of $(X,\fb_i)$, one obtains the estimate
\[
\ord_{E_i}\fa_{ij}\le\ord_{F_{ij}}\fa_{ij}\le r^{-1}l'
\]
for all $1\le j\le e_1$ along step~2 of the proof of \cite[lemma~6.11]{HLLaxv}. Set $d_{ij}=r_j-r_{ij}>0$ for $1\le j\le e_1$. We may assume that $\{d_{i1}^{-1}d_{ij}\}_i$ is non-increasing for all $j$. We set $t=r\delta/e_1l'm$ for the maximum $m$ of all $d_{11}^{-1}d_{1j}$. We may assume that $d_{11}\le t$.

Set $\fc_i=(\prod_{j\le e_1}\fa_{ij}^{s_{ij}})(\prod_{j>e_1}\fa_{ij}^{r_j})$ for $s_{ij}=r_j-d_{i1}^{-1}d_{ij}t\ge0$. Since $s_{i1}=r_1-t$ is constant, the inductive hypothesis provides a positive integer $l$ such that for all $i$, there exists a divisor $G_i$ over $X$ which computes $\mld_x(X,\fc_i)$ and has $a_{G_i}(X)\le l$. Using $\ord_{E_i}\fa_{ij}\le r^{-1}l'$, one obtains the estimate
\[
\mld_x(X,\fc_i)\le a_{E_i}(X,\fc_i)\le b+\delta
\]
along step~3 of the proof of \cite[lemma~6.11]{HLLaxv}. Then $a_{G_i}(X,\fb_i)\le a_{G_i}(X,\fc_i)\le b+\delta$ and thus $a_{G_i}(X,\fb_i)=b$. Namely, $G_i$ computes $\mld_x(X,\fb_i)$ as well as $\mld_x(X,\fc_i)$. By \cite[lemma~4.5(ii)]{K21}, $G_i$ also computes $\mld_x(X,\fa_i)$ for $\fa_i=\fb_i^{1-d_{i1}/t}\fc_i^{d_{i1}/t}$.
\end{proof}

\begin{corollary}\label{crl:quot}
The four statements in Theorem~\textup{\ref{thm:equiv}} hold for a terminal quotient threefold singularity $x\in X$.
\end{corollary}

\begin{proof}
By the reduction \cite[lemma~4.11]{K21} to $\bQ$-ideals, it suffices to find, for a fixed positive integer $n$, a positive integer $l$ depending only on $X$ and $n$ such that for any ideal $\fa$ in $\sO_X$, there exists a divisor $E$ over $X$ which computes $\mld_x(X,\fa^q)$ for $q=1/n$ and has $a_E(X)\le l$. We shall prove this by induction on the index $r$ of $x\in X$. The case $r=1$ is Theorem~\ref{thm:nakamura}. We assume that $r\ge2$. We write $\fm$ for the maximal ideal in $\sO_X$ defining $x$.

We may assume $\fa$ to be non-trivial and $\fm$-primary. Then there exists a positive rational number $q'$ such that $\mld_x(X,\fa^{q'})=1$. By \cite[lemma~5.8]{K21}, there exists a divisorial contraction $F\subset Y\to x\in X$ which contracts the divisor $F$ to the point $x$ such that $a_F(X,\fa^{q'})=1$. It is uniquely determined as in Theorem~\ref{thm:divcont}(\ref{itm:divcont-quot}). It has $a_F(X)=1+1/r$, and $Y$ has only quotient singularities of index less than $r$ as remarked after that theorem.

Suppose that $(X,\fa^q)$ is terminal or equivalently $q<q'$. We shall provide an argument independent of \cite{HLLaxv}. Take a positive rational number $\delta$ such that no integer $a$ satisfies $q<1/ar<q+\delta$. Since $q'\ord_F\fa=a_F(X)-a_F(X,\fa^{q'})=1/r$, it follows that $q+\delta\le q'$, that is, $(X,\fa^{q+\delta})$ is canonical. Let $E$ be a divisor over $X$ which computes $\mld_x(X,\fa^q)$. Then $a_E(X,\fa^q)\le\mld_xX<2$ whilst $a_E(X,\fa^{q+\delta})\ge1$. Thus
\[
a_E(X)=\Bigl(1+\frac{q}{\delta}\Bigr)a_E(X,\fa^q)-\frac{q}{\delta}a_E(X,\fa^{q+\delta})<2+\frac{q}{\delta}.
\]

Suppose that $(X,\fa^q)$ is not terminal or equivalently $q'\le q$. The log discrepancy $a_F(X,\fa^q)$ is expressed as $1-c/rn$ by a non-negative integer $c$. Let $q_1=q/r!=1/r!n$ and let $\fc=\fb\sO_Y(-(r-1)!cF)$ for the weak transform $\fb$ in $Y$ of $\fa^{r!}$. Then $(Y,\fc^{q_1})$ is crepant to $(X,\fa^q)$.

By Remark~\ref{rmk:equiv}(\ref{itm:equiv-nonpos}), we may assume that $\mld_x(X,\fa^q)$ is positive. Then by Theorem~\ref{thm:lct}, there exists a positive integer $b$ depending only on $X$ and $n$ such that $\ord_E\fm\le b$ for every divisor $E$ over $X$ that computes $\mld_x(X,\fa^q)$. Take a scheme-theoretic point $\eta$ in $F$ such that $\mld_\eta(Y,\fc^{q_1})=\mld_x(X,\fa^q)$. We shall find a positive integer $l'$ depending only on $X$, $n$ (and $Y$, $q_1$) such that there exists a divisor $E$ over $Y$ which computes $\mld_\eta(Y,\fc^{q_1})=\mld_x(X,\fa^q)$ and has $a_E(Y)\le l'$. Then
\[
a_E(X)=a_E(Y)+(a_F(X)-1)\ord_EF\le l'+\frac{1}{r}\ord_E\fm\le l'+\frac{b}{r}.
\]

Write $Z$ for the closure of $\{\eta\}$ in $Y$. If $Z$ is a divisor, then $F$ computes $\mld_\eta(Y,\fc^{q_1})$ and $a_F(Y)=1$. If $Z$ is a curve $C$, then we take the general hyperplane section $H$ of $Y$ and a point $z$ in $H\cap C$ just as in step~2 of the proof of \cite[lemma~5.17]{K21}. Considering a log resolution, one has $\mld_z(H,(\fc\sO_H)^{q_1})=\mld_\eta(Y,\fc^{q_1})$ and there exists a divisor $E$ over $Y$ with $c_Y(E)=C$ such that a component $e$ of $E\times_YH$ mapped to $z$ computes $\mld_z(H,(\fc\sO_H)^{q_1})$. By Remark~\ref{rmk:equiv}(\ref{itm:equiv-sf}), there exists a positive integer $l_1$ depending only on $q_1$ for which one may take $E$ such that $a_e(H)\le l_1$. This $E$ computes $\mld_\eta(Y,\fc^{q_1})$ and has $a_E(Y)=a_e(H)\le l_1$. Finally if $Z$ is a closed point $y$, then the inductive hypothesis provides a positive integer $l_2$ depending only on $X$, $n$ (and $Y$, $q_1$) such that there exists a divisor $E$ over $Y$ which computes $\mld_y(Y,\fc^{q_1})$ and has $a_E(Y)\le l_2$. Take $l'$ to be the maximum of $l_1$ and $l_2$.
\end{proof}

\begin{proof}[Proof of Corollary~\textup{\ref{crl:main}}]
Since the number of types of terminal quotient threefold singularities of bounded index is finite, it suffices to show the ACC for minimal log discrepancies on a fixed singularity. Thus the corollary follows from Corollary~\ref{crl:quot}.
\end{proof}


\begin{thebibliography}{99}
\bibitem{Al93}
V. Alexeev.
Two two-dimensional terminations.
Duke Math.\ J. \textbf{69} (1993), 527--545.
\bibitem{dFEM10}
T. de Fernex, L. Ein and M. Musta\c{t}\u{a}.
Shokurov's ACC conjecture for log canonical thresholds on smooth varieties.
Duke Math.\ J. \textbf{152} (2010), 93--114.
\bibitem{dFEM11}
T. de Fernex, L. Ein and M. Musta\c{t}\u{a}.
Log canonical thresholds on varieties with bounded singularities.
\textit{Classification of algebraic varieties}, 221--257.
EMS Ser.\ Congr.\ Rep., Eur.\ Math.\ Soc., 2011.
\bibitem{dFM09}
T. de Fernex and M. Musta\c{t}\u{a}.
Limits of log canonical thresholds.
Ann.\ Sci.\ \'Ec.\ Norm.\ Sup\'er.\ (4) \textbf{42} (2009), 491--515.
\bibitem{EMY03}
L. Ein, M. Musta\c{t}\u{a} and T. Yasuda.
Jet schemes, log discrepancies and inversion of adjunction.
Invent.\ Math.\ \textbf{153} (2003), 519--535.
\bibitem{HLLaxv}
J. Han, J. Liu and Y. Luo.
ACC for minimal log discrepancies of terminal threefolds.
arXiv:2202.05287.
\bibitem{Ji21}
C. Jiang.
A gap theorem for minimal log discrepancies of noncanonical singularities in dimension three.
J. Algebraic Geom.\ \textbf{30} (2021), 759--800.
\bibitem{K01}
M. Kawakita.
Divisorial contractions in dimension three which contract divisors to smooth points.
Invent.\ Math.\ \textbf{145} (2001), 105--119.
\bibitem{K13}
M. Kawakita.
Ideal-adic semi-continuity of minimal log discrepancies on surfaces.
Michigan Math.\ J. \textbf{62} (2013), 443--447.
\bibitem{K14}
M. Kawakita.
Discreteness of log discrepancies over log canonical triples on a fixed pair.
J. Algebraic Geom.\ \textbf{23} (2014), 765--774.
\bibitem{K15}
M. Kawakita.
A connectedness theorem over the spectrum of a formal power series ring.
Internat.\ J. Math.\ \textbf{26} (2015), 1550088, 27pp.
\bibitem{K17}
M. Kawakita.
Divisors computing the minimal log discrepancy on a smooth surface.
Math.\ Proc.\ Cambridge Philos.\ Soc.\ \textbf{163} (2017), 187--192.
\bibitem{K21}
M. Kawakita.
On equivalent conjectures for minimal log discrepancies on smooth threefolds.
J. Algebraic Geom.\ \textbf{30} (2021), 97--149.
\bibitem{K24}
M. Kawakita.
\textit{Complex algebraic threefolds}.
Cambridge Studies in Advanced Mathematics \textbf{209}, Cambridge University Press, 2024.
\bibitem{Ka96}
Y. Kawamata.
Divisorial contractions to $3$-dimensional terminal quotient singularities.
\textit{Higher dimensional complex varieties}, 241--246.
De Gruyter, 1996.
\bibitem{Koaxv}
J. Koll\'ar.
Which powers of holomorphic functions are integrable?
arXiv:0805.0756.
\bibitem{KM92}
J. Koll\'ar and S. Mori.
Classification of three-dimensional flips.
J. Amer.\ Math.\ Soc.\ \textbf{5} (1992), 533--703;
errata by S. Mori ibid.\ \textbf{20} (2007), 269--271.
\bibitem{KSC04}
J. Koll\'ar, K. E. Smith and A. Corti.
\textit{Rational and nearly rational varieties}.
Cambridge Studies in Advanced Mathematics \textbf{92}, Cambridge University Press, 2004.
\bibitem{LLaxv}
J. Liu and Y. Luo.
Second largest accumulation point of minimal log discrepancies of threefolds.
arXiv:2207.04610.
\bibitem{LX21}
J. Liu and L. Xiao
An optimal gap of minimal log discrepancies of threefold non-canonical singularities.
J. Pure Appl.\ Algebra \textbf{225} (2021), 106674, 23pp.
\bibitem{MN18}
M. Musta\c{t}\u{a} and Y. Nakamura.
A boundedness conjecture for minimal log discrepancies on a fixed germ.
\textit{Local and global methods in algebraic geometry}, 287--306.
Contemp.\ Math.\ \textbf{712}, Amer.\ Math.\ Soc., 2018.
\bibitem{Sh04}
V. V. Shokurov.
Letters of a bi-rationalist V: Mld's and termination of log flips.
Tr.\ Mat.\ Inst.\ Steklova \textbf{246} (2004), 328--351;
translation in Proc.\ Steklov Inst.\ Math.\ \textbf{246} (2004), 315--336.
\bibitem{St11}
D. A. Stepanov.
Smooth three-dimensional canonical thresholds.
Mat.\ Zametki \textbf{90} (2011), 285--299;
translation in Math.\ Notes \textbf{90} (2011), 265--278.
\end{thebibliography}
\end{document}